\documentclass[english,10pt]{amsart}

\usepackage{amsmath, amssymb, amsthm, enumerate}
\usepackage[all]{xy}
\usepackage[activeacute]{babel}

\newtheorem{thm}[equation]{Theorem}
\makeatletter
\let\c@subsubsection\c@equation
\makeatother
\newtheorem{lem}[equation]{Lemma}
\newtheorem{prop}[equation]{Proposition}
\newtheorem{cor}[equation]{Corollary}

\theoremstyle{remark}
\newtheorem{rmk}[equation]{Remark}

\theoremstyle{definition}
\newtheorem{defi}[equation]{Definition}
\newtheorem{exam}[equation]{Example}

\newcommand{\Hom}{\mathrm{Hom}}
\newcommand{\Gm}{\mathbb G _{m}}

\newcommand{\spec}[1]{\mathrm{Spec}(#1)}
\newcommand{\generators}{\mathcal G}
\newcommand{\generatorstot}{\mathcal K}
\newcommand{\homotcat}{\mathbf{Ho}}

\newcommand{\hr}{HR}
\newcommand{\finfield}{\mathbb F}
\newcommand{\hrfin}{H\mathbb \finfield}
\newcommand{\sphere}{\mathbf 1}
\newcommand{\ratMoore}{\sphere _{\mathbb Q}}
\newcommand{\hocolim}{\mathrm{hocolim}}

\newcommand{\beilinson}{\mathbf{H}_{\mathfrak B}}

\newcommand{\unsmot}{\mathcal M}

\newcommand{\susp}[1]{\Sigma ^{#1}}
\newcommand{\Tspectra}{Spt(\mathcal M)}
\newcommand{\stablehomotopy}{\mathcal{SH}}
\newcommand{\Tspectranbirat}[1]{B_{#1}\Tspectra}
\newcommand{\stablehomotopynbirat}[1]{\stablehomotopy(B_{#1})}
\newcommand{\Tspectranwbirat}[1]{WB_{#1}\Tspectra}
\newcommand{\stablehomotopynwbirat}[1]{\stablehomotopy(WB_{#1})}
 
\newcommand{\neffstablehomotopy}[1]{\Sigma _{T}^{#1}\stablehomotopy^{\mathit{eff}}}
\newcommand{\northogonal}[1]{\stablehomotopy ^{\perp}(#1)}
\newcommand{\stablehomotopyeff}{\stablehomotopy^{\mathit{eff}}}

\numberwithin{equation}{subsection}

\begin{document}


\title{Motivic Birational Covers and Finite Filtrations on Chow Groups}


\author{Pablo Pelaez}
\address{Instituto de Matem\'aticas, Ciudad Universitaria, UNAM, DF 04510, M\'exico}
\email{pablo.pelaez@im.unam.mx}





\begin{abstract}
	We introduce a functorial tower of localizing subcategories 
	in the Morel-Voevodsky motivic stable homotopy category.  We
	study the tower for the spectrum $HG$ representing motivic cohomology with coefficients
	in an abelian group $G$.  When $G=\mathbb Z /\ell$ is a field, we show that the tower induces
	an interesting finite filtration on the Chow groups of smooth schemes over
	a perfect field.  For $G=\mathbb Z$ or $G=\mathbb Q$, we show  that the tower induces
	an interesting finite filtration on the motivic cohomology groups of smooth schemes over
	an algebraically closed field.  With rational coefficients, this
	finite filtration satisfies several of the properties
	of the still conjectural Bloch-Beilinson-Murre filtration.
\end{abstract}

\thanks{Research partially supported by DGAPA-UNAM grant IA100814.}

\maketitle


\section{Introduction}  \label{sec.introd}

The main goal of this paper is to construct an interesting finite filtration on the Chow groups
(with coefficients in a finite field $\mathbb Z/\ell$) of algebraic cycles modulo rational equivalence, and
as well an interesting finite filtration on the motivic cohomology groups (with integral or rational coefficients).
Considering rational coefficients, we obtain a finite filtration
which satisfies several of the properties of the still conjectural Bloch-Beilinson-Murre
filtration \cite{MR923131}, \cite{MR558224}, \cite{MR1225267}.

Our approach can be sketched quickly as follows.
For a smooth scheme $Y$ of finite type over a perfect field $k$, the Chow groups can be computed
in Voevodsky's triangulated category of motives \cite{MR1883180}.  Thus, a standard adjointness argument
implies that the Chow groups can be computed as well
in the Morel-Voevodsky motivic stable homotopy category $\stablehomotopy$: 
\[CH^{q}(Y)_{R}\cong \Hom _{\stablehomotopy}(\susp{-q}(\Gm^{-q} \wedge Y_{+}), \hr);\] 
where $\hr$ represents in $\stablehomotopy$
motivic cohomology with $R$-coefficients.  Since $\stablehomotopy$ is a triangulated category,
it is possible to construct the filtration by considering a tower in $\stablehomotopy$
(see \S \ref{sec.biratHZ}):
	\[	\cdots \rightarrow wb^{c} _{-3}(\hr) \rightarrow wb^{c} _{-2}(\hr) \rightarrow wb^{c}_{-1}(\hr) 
		\rightarrow \hr \]
and defining the $p$-component of the filtration $F^{p}CH^{q}(Y)$ to be the image of the induced map:
	\[ \xymatrix@C=0.8pt{\Hom _{\stablehomotopy}(\susp{-q} (\Gm^{-q}\wedge Y_{+}), 
		wb^{c}_{-p}\hr) \ar[rr] &&
		 \Hom _{\stablehomotopy}(\susp{-q}(\Gm^{-q}\wedge Y_{+}), \hr)\cong CH^{q}(Y)_{R}}
	 \]
With this definition, it is not at all obvious that the filtration $F^{\bullet}CH^{q}(Y)$ is finite.
This is proved in \ref{thm.filt.motcoh.finite}.  It is also interesting to observe that the filtration is defined (and is finite)
for any coefficient ring $R$ (not just the rationals) and for any smooth $k$-scheme of finite type
(not necessarily projective).

Now, we describe the contents of the paper.  In \S \ref{sec.mainres}, the main results are stated and proved using the constructions
of the remaining sections.  In \S \ref{subsec.defandnots}-\ref{subsec.Quillenmodcats},
we fix the notation and introduce the basic definitions that will be used in the
rest of the paper.  In \S \ref{sec.locsubcats}, we prove some general results for localizing and orthogonal subcategories
of a triangulated category which is the homotopy category of a Quillen model category,
see \ref{thm.preslocsub}, \ref{thm.presorth}, \ref{thm.perpcompgen}, \ref{cor.perp.2prestns}.
In \S \ref{sec.birational}, we introduce the weakly birational covers and the weakly birational tower which will
be used to construct the filtrations we are interested in, see \ref{eq.birtower}, \ref{def.birat.cover},
\ref{thm.3.1.octahedral-axiom}, \ref{thm.birspecseq}.
In \S \ref{Q.coeffs}, 
we study the weakly birational tower
with rational coefficients and show that in this case, all the components of the tower are motives and the corresponding
maps in the tower are as well maps of motives, see \ref{thm.birtowQloc}.  
We also prove  some auxiliary results of independent interest
for the slice filtration with rational coefficients,
see \ref{lem.Qloccompslice}, \ref{cor.slice.HBmod}, \ref{rat.orth.motive}.
Finally, in \S \ref{sec.biratHZ}, we study the weakly birational
tower for motivic cohomology and describe the main properties of the induced filtration,
see \ref{thm.motoperspecseq}, \ref{thm.hzcoversnotriv}, \ref{thm.filt.motcohY},
\ref{thm.filt.motcoh.finite}.

\subsection{Main Results}  \label{sec.mainres}

We will consider a base scheme of the form $X=\spec k$, where
$k$ is a perfect field.  
We will write $\hrfin$ (resp. $\hrfin ^{p,q}$) for the spectrum representing in $\stablehomotopy$
motivic cohomology with coefficients in a finite field $\finfield =\mathbb Z/\ell$, where $\ell \neq \mathrm{char} (k)$;
(resp. $\susp{p-q}(\Gm ^{q} \wedge \hrfin)$).
Let $\theta _{-n\ast}^{\hrfin}:wb^{c}_{-n}\hrfin \rightarrow \hrfin$ be the universal map defined in \eqref{motbirtower}.

\begin{thm}  \label{mainthm1}
Let $Y\in Sm_{X}$, and $q\geq 0$.  Then,
there exists a non-trivial finite decreasing filtration
$F^{\bullet}$ on $CH^{q}(Y)_{\finfield}$ where the $n$-component is given by the image of 
$\theta _{-n\ast}^{\hrfin}$:
\[ \xymatrix{\Hom _{\stablehomotopy}(\susp{-q}(\Gm ^{-q}\wedge Y_{+}), wb^{c}_{-n}\hrfin) 
\ar[r]^-{\theta _{-n\ast}^{\hrfin}} &
\Hom _{\stablehomotopy}(\susp{-q}(\Gm ^{-q}\wedge Y_{+}), \hrfin)=CH^{q}(Y)_{\finfield}.} \]
In addition, the filtration $F^{\bullet}$ is functorial in $Y$ with respect to morphisms in $Sm_{X}$ and
satisfies:
\begin{enumerate}
		\item  $F^{0}CH^{q}(Y)_{\finfield}=CH^{q}(Y)_{\finfield}$,
		\item  $F^{q+1}CH^{q}(Y)_{\finfield}=0$.
	\end{enumerate}
\end{thm}
\begin{proof}
Since $CH^{q}(Y)_{\finfield}\cong \Hom _{\stablehomotopy}(\susp{-q}(\Gm^{-q} \wedge Y_{+}), \hrfin)$,
we deduce all the properties (except for the non-triviality)
by combining \ref{thm.filt.motcohY} and \ref{thm.filt.motcoh.finite}.  

To show that the filtration is in general
non-trivial, we will assume that $q=\ell m$ for some $m\geq 1$, and that there exists 
$a\in CH^{m}(Y)_{\finfield}$ such that $0\neq a^{\ell}\in CH^{q}(Y)_{\finfield}$.
Hence by \cite[Lemma 9.8]{MR2031198}, we deduce that $0\neq a^{\ell}=P^{m}(a)\in CH^{q}(Y)_{\finfield}$,
where $P^{m}: \hrfin \rightarrow \hrfin ^{2m(\ell-1),m(\ell-1)}$ is the motivic operation $P^{m}$
constructed by Voevodsky in \cite[p. 33]{MR2031198}.  By adjointness, $P^{m}(a)$ is the following composition
in $\stablehomotopy$:
\[ \xymatrix{ \susp{-m(\ell-1)} (\Gm ^{-m(\ell-1)}\wedge \susp{-m}(\Gm ^{-m}\wedge Y_{+})) 
\ar[d]^-{\susp{-m(\ell-1)} (\Gm ^{-m(\ell-1)}\wedge a)}\\ \susp{-m(\ell-1)} (\Gm ^{-m(\ell-1)}\wedge \hrfin) 
=\hrfin ^{-2m(\ell-1),-m(\ell-1)}
\ar[d]^-{\susp{-m(\ell-1)} (\Gm ^{-m(\ell-1)}\wedge P^{m})}\\ \hrfin}
\] 
Finally, by \ref{negoperbirat} we deduce that $0\neq P^{m}(a)\in F^{m(\ell-1)}CH^{q}(Y)_{\finfield}$.
\end{proof}

Let $Y\in Sm_{X}$.  We will write $H^{\ast,\ast}(Y,R)$ for $\oplus _{p,q\in \mathbb Z}H^{p,q}(Y, R)$,
where:
\[	H^{p,q}(Y,R)=\Hom _{\stablehomotopy}(\susp{-p+q}(\Gm ^{-q}\wedge Y_{+}),HR)
\] 
is the motivic cohomology of $Y$ of degree $p$ and weight $q$, with $R$-coefficients.
In the next result we will consider motivic cohomology with integer and rational coefficients.
We will write $R$ for $\mathbb Z$ or $\mathbb Q$; and $\sphere$ for $\Sigma _{T}^{\infty}X_{+}$,
the sphere spectrum in $\stablehomotopy$.

\begin{thm}  \label{mainthm2}
Let the base scheme $X=\spec k$, with $k$ an algebraically closed field; and $Y\in Sm_{X}$.  
Then, there exists a non-trivial decreasing filtration
$F^{\bullet}$ on $H^{\ast,\ast}(Y,R)$, where the $n$-component in degree $p$
and weight $q$ is given by the image of 
$\theta _{-n\ast}^{\hr}$:
\[ \xymatrix{\Hom _{\stablehomotopy}(\susp{-p+q}(\Gm ^{-q}\wedge Y_{+}), wb^{c}_{-n}\hr) 
\ar[r]^-{\theta _{-n\ast}^{\hr}} &
\Hom _{\stablehomotopy}(\susp{-p+q}(\Gm ^{-q}\wedge Y_{+}), \hr) \ar@{=}[d]\\
& H^{p,q}(Y,R).} \]
In addition, the filtration $F^{\bullet}$ is functorial in $Y$ with respect to morphisms in $Sm_{X}$ and
satisfies:
\begin{enumerate}
		\item  $F^{0}H^{\ast,\ast}(Y,R)=H^{\ast,\ast}(Y,R)$,
		\item  $F^{q+1}H^{p,q}(Y,R)=0$.
	\end{enumerate}
\end{thm}
\begin{proof}
Combining \ref{thm.filt.motcohY} and \ref{thm.filt.motcoh.finite},
we deduce all the properties, except for the non-triviality of the filtration.

To show the non-triviality of the filtration, by \ref{thm.hzcoversnotriv} and \ref{adj.motoper}
it suffices to construct a non-trivial map $HR\rightarrow \Gm ^{q}\wedge HR$, for $q>0$.  
To obtain such a map,
it suffices to show the existence of a non-trivial map $a: \sphere \rightarrow \Gm ^{q}\wedge HR$; since
$a\wedge HR$ will produce the non-trivial map we were looking for.

By definition of the sphere spectrum, $\Hom _{\stablehomotopy}(\sphere, \Gm ^{q}\wedge HR)\cong
H^{q,q}(X,R)$, where $X=\spec k$.  Now, a theorem of Nesterenko-Suslin and Suslin-Voevodsky
\cite{MR992981}, \cite[Thm. 3.4]{MR1744945} implies that $H^{q,q}(X,R)$ is isomorphic
to the degree $q$ Milnor $K$-theory of $k$ with $R$ coefficients, $K^{M}_{q}(k)_{R}$.
However, it is well-known that $K^{M}_{q}(k)_{\mathbb Z}$ is uniquely divisible
\cite[Ch. III, \S 7, Ex. 7.3(a)]{MR3076731}.  Thus,
$H^{q,q}(X,\mathbb Z)\neq 0$ and $H^{q,q}(X,\mathbb Q)\neq 0$, as we wanted.
\end{proof}

With rational coefficients, the filtration in \ref{mainthm2} is not only functorial with respect to
maps in $Sm_{X}$, but in fact it is functorial with respect to Suslin-Voevodsky transfers.
Let $DM_{\mathbb Q}$ 
denote Voevodsky's big (admitting infinite direct sums) triangulated category of motives with rational coefficients.

\begin{thm}  \label{mainthm3}
Let $R=\mathbb Q$, and $X=\spec k$ with $k$ a perfect field.  Then the filtration $F^{\bullet}$ in \ref{mainthm2} satisfies:
	\begin{enumerate}
		\item	\label{mainthm3.1}  $F^{\bullet}$ is functorial with respect to maps in $DM_{\mathbb Q}$.
		\item \label{mainthm3.2}  $F^{\bullet}$ is functorial with respect to Suslin-Voevodsky transfers.
	\end{enumerate}
\end{thm}
\begin{proof}
It suffices to prove \eqref{mainthm3.1}, since
there is a functor $Cor(X)\rightarrow DM_{\mathbb Q}$ \cite[Thm 3.2.6]{MR1764202}, where
$Cor(X)$ is the Suslin-Voevodsky category of finite correspondences over $X$; 
having same objects as $Sm_{X}$, morphisms $c(Y,Z)$ given by the group of finite relative cycles on
$Y\times _{X}Z$ over $Y$ \cite{MR1764199}, and composition as in 
\cite[(2.1) p.\,673]{MR2804268}.

Finally, \eqref{mainthm3.1} follows directly from \ref{thm.birtowQloc}.
\end{proof}

When $k$ is an algebraically closed field (e.g. $k=\mathbb C$), combining \ref{mainthm2} and \ref{mainthm3}
we deduce that the filtration $F^{\bullet}$ with rational coefficients
satisfies several of the properties of the still conjectural
Bloch-Beilinson-Murre filtration \cite{MR923131}, \cite{MR558224}, \cite{MR1225267}, {\cite[\S 2]{MR1265533}},
which was originally defined for the Chow groups with rational coefficients, 
but admits a natural extension to motivic cohomology of
arbitrary degree and weight.  Namely, our filtration $F^{\bullet}$ is non-trivial, finite and compatible with the action
of correspondences (since it admits Suslin-Voevodsky transfers).  Unfortunately, the methods of this paper are not
sufficient to show that the filtration is non-trivial when restricted to the Chow groups with rational coefficients.

\subsection{Definitions and Notation}	\label{subsec.defandnots}		
	In this paper $X$ will denote a Noetherian separated base scheme of finite Krull dimension,
	$Sch_{X}$ the category of schemes of finite type over $X$ and $Sm_{X}$ the full
	subcategory of $Sch_{X}$ consisting of smooth schemes over $X$ regarded
	as a site with the Nisnevich topology.	  All the maps between schemes will be considered over
	the base $X$.  Given $Y\in Sch_{X}$, all the closed subsets $Z$ of $Y$ will be considered
	as closed subschemes with the reduced structure.
	
	Let  $\unsmot$ be the category of pointed simplicial presheaves on $Sm_{X}$
	equipped with the motivic Quillen model structure \cite{MR0223432} constructed by Jardine 
	\cite[Thm.\,1.1]{MR1787949}, which is Quillen equivalent to the one defined originally by
	Morel-Voevodsky \cite[Thm.\,1.2]{MR1787949}, taking the affine line $\mathbb A _{X}^{1}$ as interval 
	\cite[p.\,86 Thm.\,3.2]{MR1813224}.
	Given a map $f:W\rightarrow Y$ in $Sm_{X}$, we will abuse notation and denote
	by $f$ the induced map $f:W_{+}\rightarrow Y_{+}$ in $\mathcal M$ between the corresponding pointed
	simplicial presheaves represented by $W$ and $Y,$ respectively.
	
	We define $T$ in 
	$\unsmot$ as the pointed simplicial presheaf represented by 
	$S^{1}\wedge \mathbb G_{m}$, where $\mathbb G_{m}$ is the multiplicative group 
	$\mathbb A^{1}_{X}-\{ 0 \}$ pointed by $1$, and $S^{1}$ denotes the simplicial circle.
	Given an arbitrary integer $r\geq 1$, let  $S^{r}$ (resp. $\Gm ^{r}$) denote the
	iterated smash product of $S^{1}$ (resp. $\Gm$) with $r$-factors: $S^{1}\wedge \cdots \wedge S^{1}$
	(resp. $\Gm \wedge \cdots \wedge \Gm$);
	$S^{0}= \Gm ^{0}$ will be by definition equal to the pointed simplicial presheaf represented by the base scheme $X$.
	Let $\Tspectra$ denote Jardine's category of symmetric $T$-spectra on 
	$\unsmot$ equipped with the motivic model structure defined in 
	\cite[Thm. 4.15]{MR1787949}.  The smash product of symmetric $T$-spectra endows
	$\Tspectra$ with the structure of a symmetric monoidal model category \cite[Prop. 4.14]{MR1787949}.  Hence, it is possible
	to consider (commutative) ring spectra with unit in $\Tspectra$ and modules in $\Tspectra$ over
	a given (commutative) ring spectrum.
	
We will use the language of triangulated categories.  Our main reference will 
be \cite{MR1812507}.  Given a triangulated category, we will write $\susp{1}$ 
(resp. $\susp{-1}$) to denote its suspension 
(resp. desuspension) functor; and for $n>0$, $\susp{n}$ (resp. $\susp{-n}$)
will be the composition of $\susp{1}$
(resp. $\susp{-1}$) iterated $n$-times  (resp. $-n$-times).  To simplify
the notation, we will write  $\susp{0}$ for the identity functor.

We will use the following notation in all the categories under consideration: $\ast$ will
	denote the terminal object, and $\cong$ will denote that a map or a functor is an isomorphism.
If a category is simplicial, we will write $Map(-,-)$ for its simplicial set of morphisms.

\subsection{Quillen model categories \cite{MR0338129}}  \label{subsec.Quillenmodcats}
	We will follow the conventions
	of \cite{MR1944041}, in particular all the Quillen model categories that we will
	consider will be closed under arbitrary limits and colimits.
	Given a Quillen model category $\mathcal A$, unless explicitly stated otherwise we will write $\homotcat (\mathcal A)$ for 	its 
	homotopy category.

\subsubsection{Bousfield localization}  \label{subsubsec.Bousloc}
For details and definitions about Bousfield localization we refer the reader to Hirschhorn's book \cite{MR1944041}.
	Let us just mention the following theorem of Hirschhorn, which guarantees the existence of left and right Bousfield localizations.
	
\begin{thm}[{see \cite[Thms. 4.1.1 and 5.1.1]{MR1944041}}]  \label{thm.Hirsch-Bousloc}
	Let $\mathcal A$ be a Quillen model category which is cellular and proper.  Let $L$ be a set of maps in $\mathcal A$ and
	let $K$ be a set of objects in $\mathcal A$.
	Then:
	\begin{enumerate}
		\item	The left Bousfield localization of $\mathcal A$ with respect to $L$ exists.
		\item	The right Bousfield localization of $\mathcal A$ with respect to the class of $K$-colocal
				equivalences exists.
	\end{enumerate}
\end{thm}

The following result guarantees the existence of left and right Bousfield localizations for the motivic stable homotopy
	category $\Tspectra$.
	
\begin{thm}  \label{thm.Tspectra-cellular}
	The Quillen model category $\Tspectra$ is:
		\begin{enumerate}  
			\item  \label{thm.Tspectra-cellular.a}  \emph{cellular} (see  \cite{MR1860878},
							\cite[Cor. 1.6]{MR2197578} or \cite[Thm. 2.7.4]{MR2807904}).
			\item  \label{thm.Tspectra-celllular.b}  \emph{proper} (see \cite[Thm. 4.15]{MR1787949}).
		\end{enumerate}
\end{thm}

\section{Localizing and orthogonal subcategories}  \label{sec.locsubcats}

	Most of the results in this section could be stated in the language of triangulated categories \cite[Ch. 8-9]{MR1812507}.
	However, we will work within the framework of Quillen model categories which has the advantage that in the applications
	to motivic homotopy theory, the constructions admit a natural extension 
	to the motivic unstable homotopy category \cite{Pelaezunstsl}.
	
	In this section, $\mathcal A$ will be a simplicial \cite[II.1, II.2]{MR0223432}, stable \cite[Ch. 7]{MR1650134}, proper and 
	cellular Quillen model category.  In particular, by \ref{thm.Hirsch-Bousloc} all the left and right Bousfield
	localizations that we will consider in this section exist.
	We will write $R$ (resp. $Q$) for a fibrant (resp. cofibrant) replacement functor in $\mathcal A$.
	
	Let $\mathcal T$ denote the homotopy category of $\mathcal A$.  We will consider $\mathcal T$ as a triangulated
	category where the distinguished triangles are given by the cofibre sequences \cite[I.3]{MR0223432} in $\mathcal A$ and the
	suspension functor is induced by smashing with respect to the simplicial circle $S^{1}$ \cite[Def. 6.1.1]{MR1650134}.

\subsection{Localizing subcategories}  \label{subsec.locsubcats}
	Let $\mathcal T '$ be a subcategory of $\mathcal T$.  Recall that $\mathcal T '$ is a localizing subcategory of 
	$\mathcal T$; if
	$\mathcal T '$ is a full triangulated subcategory of $\mathcal T$, and
	$\mathcal T '$ is closed under arbitrary coproducts.

\subsubsection{Generators}

	Let $\generators$ be a class of objects in $\mathcal T=\homotcat (\mathcal A)$.  We will follow the notation of
	Neeman \cite{MR2794632} and write
	\begin{align}  \label{eq.localizing.gen}
		Loc(\generators)
	\end{align}
	for the intersection of all the localizing subcategories of $\mathcal T$ which contain
	$\generators$.  We will say that $Loc(\generators)$ is the localizing subcategory of $\mathcal T$
	generated by $\generators$.
	
\begin{lem}   \label{lem.gens.gen}
	Let $K$ be an object in $Loc(\generators)$ such that for every $G\in \generators$ and every $n\in \mathbb Z$:
	$ \Hom _{Loc(\generators)}(\susp{n}G, K)=0$. 
	Then $K=\ast$.
\end{lem}
\begin{proof}
	Let $^{\perp} K$ denote the full subcategory  of $\mathcal T$ which consists of objects $L\in \mathcal T$
	satisfying $\Hom _{\mathcal T}(\susp{n}L, K)=0$, for every $n\in \mathbb Z$.
	We observe that $^{\perp}K$ is triangulated
	and closed under arbitrary coproducts.  By hypothesis $\generators \subseteq ^{\perp}\! \! K$,
	hence $Loc(\generators)\subseteq ^{\perp}\! \! K$.  Thus, we conclude that
	$K=\ast$.
\end{proof}

\begin{cor}  \label{cor.gens.gen}
	Let $f:M\rightarrow N$ be a map in $Loc(\generators)$ such that for every $G\in \generators$ and every
	 $n\in \mathbb Z$, the induced map of abelian groups:
	 	\begin{align*}
			f_{\ast}:\Hom _{Loc(\generators)}(\susp{n}G, M)\rightarrow \Hom _{Loc(\generators)}(\susp{n}G,N)
		\end{align*}
	is an isomorphism.  Then $f$ is an isomorphism.
\end{cor}
\begin{proof}
	Complete $f$ to a distinguished triangle in $Loc(\generators)$: $M\rightarrow N\rightarrow K$.  Since the
	maps $f_{\ast}$ above are all isomorphisms, we conclude that for every $G\in \generators$ and every
	 $n\in \mathbb Z$:
			$\Hom _{Loc(\generators)}(\susp{n}G, K)=0$.
	By \ref{lem.gens.gen}, we deduce that $K\cong \ast$ in $Loc(\generators)$.  Thus $f:M\rightarrow N$
	is an isomorphism.
\end{proof}
	
\subsubsection{Presentation of localizing subcategories}  \label{subsubsec.presloc}

	From now on we will assume that $\generators$ in \eqref{eq.localizing.gen} is a set.  Let 
	 $\mathbf R _{\generators}\mathcal A$ denote the right Bousfield localization  of
	 $\mathcal A$ with respect to the set:
	 \begin{align} \label{eq.stable.generators}
		\susp{\pm \infty}\generators =\{ \susp{n}G: G\in \generators;  n\in \mathbb Z \}
	\end{align}

\begin{rmk}  \label{rmk.colobj.cof}
	By
	\cite[5.1.1.(1)(a) and 3.1.8]{MR1944041} we can assume that the objects in $\susp{\pm \infty}G$ are cofibrant in $\mathcal A$.
\end{rmk}

	Let $C$ be a functorial cofibrant replacement in $\mathbf R _{\generators}\mathcal A$.  Recall that $\mathcal A$ is a 
	simplicial Quillen model category with  fibrant replacement functor $R$.

\begin{prop}  \label{prop.homot.triang}
	The homotopy category of $\mathbf R _{\generators}\mathcal A$, $\homotcat(\mathbf R _{\generators}\mathcal A)$
	is a triangulated category and the Quillen adjunction 
	$(id, id, \varphi): \mathbf R _{\generators}\mathcal A \rightarrow \mathcal A$ 
	induces an adjunction between triangulated functors:
	\begin{align*}
		(C,R,\varphi ): \homotcat (\mathbf R _{\generators}\mathcal A) \rightarrow \mathcal T = \homotcat (\mathcal A) 
	\end{align*}
\end{prop}
\begin{proof}
	By construction,
	the suspension functor $\susp{1}$ of $\homotcat (\mathcal A)$ is induced by the Quillen adjunction
	$(S^{1},\Omega _{S^{1}},\varphi):\mathcal A \rightarrow \mathcal A$.  Since $\susp{\pm \infty}\generators=\susp{1}(\susp{\pm \infty}\generators)$,
	it follows from \cite[3.3.20(2)(b)]{MR1944041} that
	$(S^{1},\Omega _{S^{1}},\varphi):\mathbf R _{\generators}\mathcal A \rightarrow \mathbf R _{\generators}\mathcal A$
	is a Quillen equivalence.  Thus, $\mathbf R _{\generators}\mathcal A$ is a stable model category and
	its homotopy category $\homotcat(\mathbf R _{\generators}\mathcal A)$ is a triangulated category.
		
	Since the identity functor $id: \mathcal A\rightarrow \mathbf R _{\generators}\mathcal A$ is a right Quillen functor,
	we conclude that $(C,R,\varphi )$ is an adjunction.  By \cite[Props. 6.4.1 and 7.1.12]{MR1650134} we
	deduce that $C$ and $R$ are triangulated functors.
\end{proof}


\begin{prop}  \label{prop.gens.gen}
	Let $f:M\rightarrow N$ be a map in $\mathbf R _{\generators}\mathcal A$.  Then $f$ is a weak equivalence in $\mathbf R _{\generators}\mathcal A$
	if and only if  for every $G\in \generators$ and every
	 $n\in \mathbb Z$, the induced map of abelian groups:
	 	\begin{align*}
			f_{\ast}:\Hom _{ \homotcat (\mathcal A)}(\susp{n}G, M)\rightarrow 
			\Hom _{ \homotcat (\mathcal A)}(\susp{n}G,N)
		\end{align*}
	is an isomorphism. 
\end{prop}
\begin{proof}
	($\Rightarrow$): Assume that $f$ is a weak equivalence in $\mathbf R _{\generators}\mathcal A$, equivalently
	a $\susp{\pm \infty}\generators$-colocal equivalence in $\mathcal A$ \cite[5.1.1(1)(a)]{MR1944041}.  By
	\ref{rmk.colobj.cof},
	we can assume that all the objects in $\susp{\pm \infty}\generators$ are cofibrant in $\mathcal A$.  Hence, 
	for every $H$ in $\susp{\pm \infty}\generators$
	the following maps are weak equivalences of simplicial sets \cite[3.1.8]{MR1944041}:
		\[ \xymatrix{Map(H , RX) \ar[r]^-{(Rf)_{\ast}} & 
								Map(H , RY)}
		\]
	Since
	$\mathcal A$ is a simplicial model category and $H$ is cofibrant,
	we deduce that $Map(H , RX)$ and $Map(H , RY)$ are both Kan complexes.  Thus, for every $r\geq 0$  
	and every $H\in \susp{\pm \infty}\generators$ the following  diagram commutes,
	where the top row and the vertical maps are all isomorphisms of abelian groups:
		\[ \xymatrix{\pi_{r}Map(H , RX) \ar[r]^-{(Rf)_{\ast}}_-{\cong} \ar[d]_-{\cong} 
								& \pi_{r}Map(H , RY) \ar[d]^-{\cong}\\
								\Hom _{\homotcat (\mathcal A)}(\susp{r}H ,X) \ar[r]_-{f_{\ast}}
								& \Hom _{\homotcat (\mathcal A)}(\susp{r}H ,Y)}
		\]
	Therefore, for every $G\in \generators$ (see \ref{eq.stable.generators}) and every $n\in \mathbb Z$ the map:
		\[ \xymatrix{\Hom _{\homotcat (\mathcal A)}(\susp{n}G ,X) \ar[r]_-{f_{\ast}}& \Hom _{\homotcat (\mathcal A)}(\susp{n}G ,Y)}
		\]
	is an isomorphism of abelian groups.
	
	($\Leftarrow$):  Fix $H \in \susp{\pm \infty}\generators$.  Let
	$\omega _{0}$, $\eta _{0}$ denote the base points of
	$Map(H ,R)$ and $Map(H , RY)$ respectively.
	It suffices to show that the map:
		\[ \xymatrix{Map(H , RX) \ar[r]^-{(Rf)_{\ast}}& 
								Map(H , RY)}
		\]
	is a weak equivalence of simplicial sets \cite[5.1.1(1)(a) and 3.1.8]{MR1944041}.  
	
	Since $\mathcal A$ is a pointed simplicial model category,
	lemma 6.1.2 in \cite{MR1650134} 
	implies that for $k\geq 0$ the following diagram commutes:
		\[ \xymatrix{\pi_{k, \omega _{0}}Map(H , RX) \ar[r]^-{(Rf)_{\ast}}\ar[d]_-{\cong}  
								& \pi_{k, \eta _{0}}Map(H , RY)\ar[d]^-{\cong}\\
								\Hom _{\homotcat (\mathcal A)}(\susp{k}H, X) \ar[r]^-{f_{\ast}}
								& \Hom _{\homotcat (\mathcal A)}(\susp{k}H, Y) }
		\]
	By hypothesis the bottom row is an isomorphism of abelian groups, thus
	we deduce that all the maps in the top row are also isomorphisms.  Therefore,
	for every $H \in \susp{\pm \infty}\generators$, the induced map of simplicial sets:
		\[ \xymatrix{Map(H , RX) \ar[r]^-{(Rf)_{\ast}}&
								Map(H , RY)}
		\] 
	is a weak equivalence when it is restricted to the path component of $Map(H , RX)$
	containing $\omega _{0}$.  We observe that $\susp{-1}H$ is also in $\susp{\pm \infty}\generators$, hence
	the following map:
		\[ \xymatrix{Map(S^{1},Map(\susp{-1}H, RX)) 
								\ar[r] &
								Map(S^{1},Map(\susp{-1}H, RY))}
		\]
	is a weak equivalence of simplicial sets, since taking $S^{1}$-loops kills the path components that do not
	contain the base point.  Then, we observe that the rows in the following
	commutative diagram are isomorphisms, since $\mathcal A$ is a simplicial stable model category:
		\[ \xymatrix{Map(S^{1},Map(\susp{-1}H, RX)) \ar[r]^-{\cong} 
								\ar[d] & Map(H, RX) \ar[d]^-{(Rf)_{\ast}}\\
								Map(S^{1},Map(\susp{-1}H, RY)) \ar[r]_-{\cong} & Map(H, RY)}
		\]
	Hence, the two out of three property for weak equivalences implies that the $(Rf)_{\ast}$
	is a weak equivalence of simplicial sets.  
\end{proof}

\begin{thm}  \label{thm.preslocsub}
	Assume that all the objects in $\generators$ are compact in $Loc(\generators)$ in the sense of Neeman
	\cite[Def. 1.6]{MR1308405}.  Then,  
	$Loc(\generators)$ is naturally equivalent as a triangulated category to the homotopy category 
	$\homotcat (\mathbf R _{\generators}\mathcal A)$.
\end{thm}
\begin{proof}
	Combining \ref{lem.gens.gen} with the compactness of the objects in $\generators$,
	we conclude that $Loc(\generators)$ is compactly generated in the sense of Neeman \cite[Def. 1.7]{MR1308405}.
	Thus, by Neeman's version of Brown representability \cite[Thm. 4.1]{MR1308405}, we deduce that the inclusion
	$i: Loc(\generators)\rightarrow \mathcal T$ admits a right adjoint $r$.  
	
	Let $\eta$ denote the counit of the adjunction
	$(i,r,\varphi ):Loc(\generators)\rightarrow \mathcal T$, and let
	$K\in \mathbf R _{\generators}\mathcal A$ be an arbitrary cofibrant and fibrant object.    We
	observe that for every $G\in \generators$ and every $n\in \mathbb Z$, $\susp{n}G$ is in $Loc(\generators)$. Since $i$ is a 
	full embedding, we deduce:
		\[ \Hom _{ \homotcat (\mathcal A)}(\susp{n}G, i\circ r(K)) =
			\Hom _{\homotcat (\mathcal A)}(i(\susp{n}G), i\circ r(K)) =
			\Hom _{ Loc(\generators)}(\susp{n}G, r(K))
		\]
	Thus,
	by adjointness: 
		\begin{align*} 
			\Hom _{ \homotcat (\mathcal A)}(\susp{n}G, i\circ r(K)) & = \Hom _{ Loc(\generators)}(\susp{n}G, r(K)) \\
				& \cong \Hom _{ \homotcat (\mathcal A)}(i(\susp{n}G), K) = 
					\Hom _{\homotcat (\mathcal A)}(\susp{n}G, K)
		\end{align*}
	Hence, we deduce that for every $G\in \generators$, and every 
	 $n\in \mathbb Z$, the induced map of abelian groups:
	 	\begin{align*}
			\eta ^{K}_{\ast}:\Hom _{ \homotcat (\mathcal A)}(\susp{n}G, i\circ r(K))\rightarrow 
			\Hom _{ \homotcat (\mathcal A)}(\susp{n}G,K)
		\end{align*}
	is an isomorphism.  Thus, by \ref{prop.gens.gen} we conclude that the
	map $R(\eta ^{K}):R(i\circ r(K))\rightarrow K$ in $\homotcat (\mathbf R _{\generators}\mathcal A)$ is an isomorphism.
	This implies that the triangulated functor 
	$R\circ i:Loc(\generators) \rightarrow \homotcat (\mathbf R _{\generators}\mathcal A)$ is essentially surjective on objects.
	Thus, it only remains to check that
	$R\circ i$ is fully faithful.
	
	By \ref{rmk.colobj.cof}, we can assume that the objects in $\susp{\pm \infty}\generators$ are cofibrant in
	$\mathcal A$.  Thus, it follows from \cite[5.1.6]{MR1944041}  that the objects in $Loc(\generators)$
	are all cofibrant in $\mathbf R _{\generators}\mathcal A$.  Then, by
	\cite[3.5.2(2)]{MR1944041} we deduce that $R\circ i$ is fully faithful, since all the objects in $Loc(\generators)$
	are cofibrant in $\mathbf R _{\generators}\mathcal A$ and hence colocal in $\mathcal A$ 
	\cite[5.1.5 and 5.1.6]{MR1944041}.  	
\end{proof}

\begin{cor} \label{cor.hocatcompgen}
	Assume that all the objects in $\generators$ are compact in $Loc(\generators)$ in the sense of Neeman
	\cite[Def. 1.6]{MR1308405}.  Then,  $\homotcat (\mathbf R _{\generators}\mathcal A)$ is a compactly generated
	triangulated category in the sense of Neeman \cite[Def. 1.6]{MR1308405}, with $\susp{\pm \infty}\generators$
	(see \ref{eq.stable.generators}) as set of compact generators \cite[Def. 1.8]{MR1308405}.
\end{cor}
\begin{proof}
	By \ref{thm.preslocsub} it suffices to show that $Loc(\generators)$ is compactly generated with 
	set of generators given by $\susp{\pm \infty}\generators$.  But this follows directly from \ref{lem.gens.gen} since we
	are assuming that $\generators$ is a set, and that all its objects are compact.
\end{proof}

\subsection{Orthogonal subcategories}  \label{subsec.orth}

Given a triangulated subcategory $\mathcal T'$ of $\mathcal T$, let $\mathcal T'^{\perp}$ (resp. $^{\perp}\! \mathcal T'$)
denote the full subcategory of $\mathcal T$ with objects $K$ such that for every $A\in \mathcal T'$,
$\Hom _{\mathcal T}(A,K)=0$ (resp. $\Hom _{\mathcal T}(K,A)=0$).  

\begin{lem}  \label{lem.orttriang}
	Let $\mathcal T'$ be triangulated subcategory of $\mathcal T$, then $\mathcal T'^{\perp}$ and $^{\perp}\! \mathcal T'$ 
	are full triangulated subcategories of $\mathcal T$.
\end{lem}
\begin{proof}
	By definition $\mathcal T'^{\perp}$ and $^{\perp}\! \mathcal T'$ are full subcategories of $\mathcal T$.  Thus, it suffices to
	check that they are triangulated.  But this
	follows immediately from the fact that the functor $\Hom _{\mathcal T}(A,-)$ (resp. $\Hom _{\mathcal T}(-,A)$)
	is homological (resp. cohomological) for every $A\in \mathcal T$ (see \cite[Def. 1.1.7]{MR1812507}).
\end{proof}

\subsubsection{Presentation of orthogonal subcategories}  \label{subsubsec.presorth}

	From now on we will assume that $\generators$ in \eqref{eq.localizing.gen} is a set.  Our goal is to find a presentacion
	for $Loc(\generators)^{\perp}$.  Let 
	 $\mathbf L _{\generators}\mathcal A$ denote the left Bousfield localization  of
	 $\mathcal A$ with respect to the set of maps:
	 \begin{align} \label{eq.stable.maps}
		M_{\generators} =\{ \susp{n}G\rightarrow \ast: G\in \generators, n\in \mathbb Z \}
	\end{align}

\begin{rmk}  \label{rmk.obj.cof}
	Taking $F=U=id$ in
	\cite[3.3.20.(1)(b)]{MR1944041} we can assume that the objects $\susp{n}G$ appearing as domains in
	$M_{\generators}$ are cofibrant in $\mathcal A$.
\end{rmk}
	
	Let $L$ be a functorial fibrant replacement in $\mathbf L _{\generators}\mathcal A$.
	Recall that $\mathcal A$ is a simplicial Quillen model category with cofibrant replacement functor $Q$.  

\begin{prop}  \label{prop.leftloctriang}
	The homotopy category of $\mathbf L _{\generators}\mathcal A$, $\homotcat(\mathbf L _{\generators}\mathcal A)$
	is a triangulated category and the Quillen adjunction 
	$(id, id, \varphi):\mathcal A \rightarrow \mathbf L _{\generators}\mathcal A$ induces an adjunction between triangulated functors:
	\begin{align*}
		(Q,L,\varphi ):\mathcal T = \homotcat (\mathcal A) \rightarrow \homotcat (\mathbf L _{\generators}\mathcal A)
	\end{align*}
\end{prop}
\begin{proof}
	By construction,
	the suspension functor $\susp{1}$ of $\homotcat (\mathcal A)$ is induced by the Quillen adjunction
	$(S^{1},\Omega _{S^{1}},\varphi):\mathcal A \rightarrow \mathcal A$.  Since $M_{\generators}=\susp{1}(M_{\generators})$,
	it follows from \cite[3.3.20(1)(b)]{MR1944041} that
	$(S^{1},\Omega _{S^{1}},\varphi):\mathbf L _{\generators}\mathcal A \rightarrow \mathbf L _{\generators}\mathcal A$
	is a Quillen equivalence.  Thus, $\mathbf L _{\generators}\mathcal A$ is a stable model category and
	its homotopy category $\homotcat(\mathbf L _{\generators}\mathcal A)$ is a triangulated category.
	
	Since the identity functor $id: \mathcal A\rightarrow \mathbf L _{\generators}\mathcal A$ is a left Quillen functor,
	we conclude that $(Q,L,\varphi )$ is an adjunction.  By \cite[Props. 6.4.1 and 7.1.12]{MR1650134} we
	deduce that $Q$ and $L$ are triangulated functors.
\end{proof}

\begin{exam}  \label{exam.invmaps}
	Let $W$ be a set consisting of maps in $\mathcal A$. 
	Let 
	 $\mathbf L _{W}\mathcal A$ denote the left Bousfield localization  of
	 $\mathcal A$ with respect to the set of maps:
	 \begin{align*} 
		sW=\{ \susp{n}f: f\in W, n\in \mathbb Z\}
	\end{align*}
	$\mathbf L _{W}\mathcal A$ admits a description as in \ref{subsubsec.presorth}, \ref{eq.stable.maps}.  For this,
	consider the set of objects $\generators _{W}$ in $\mathcal A$ which complete the maps in $W$ to cofibre sequences in $\mathcal A$ 
	(distinguished triangles in $\homotcat (\mathcal A)$).  Namely, there is a bijection of sets $W\rightarrow \generators _{W}$, $f\mapsto G_{f}$
	such that for every $f:E\rightarrow F$ in $W$, $G_{f}$ fits in a cofibre sequence in $\mathcal A$:
	$E \rightarrow F \rightarrow G_{f}$.  Let $\mathbf L _{\generators _{W}}\mathcal A$
	be the left Bousfield localization of $\mathcal A$ with respect to the set of maps $M_{\generators _{W}}$ (see \ref{eq.stable.maps}).
\end{exam}

	Let $L$ be a functorial fibrant replacement in $\mathbf L _{W}\mathcal A$.
	Recall that $\mathcal A$ is a simplicial Quillen model category with cofibrant replacement functor $Q$.  

\begin{lem}  \label{lem.leftmapstriang}
	\begin{enumerate}
		\item \label{lem.leftmapstriang.a}  The homotopy category of $\mathbf L _{W}\mathcal A$, $\homotcat(\mathbf L _{W}\mathcal A)$
	is a triangulated category and the Quillen adjunction 
	$(id, id, \varphi):\mathcal A \rightarrow \mathbf L _{W}\mathcal A$ induces an adjunction between triangulated functors:
	\begin{align*}
		(Q,L,\varphi ):\mathcal T = \homotcat (\mathcal A) \rightarrow \homotcat(\mathbf L _{W}\mathcal A)
	\end{align*}
		\item \label{lem.leftmapstriang.d}  Let $G\in \generators _{W}$ and $\susp{n}G\rightarrow \ast$ a map in $M_{\generators _{W}}$ 
		(see \ref{eq.stable.maps}).  Then $\susp{n}G\rightarrow \ast$ becomes an isomorphism in 
		$\homotcat (\mathbf L _{W}\mathcal A)$.
		\item \label{lem.leftmapstriang.b} Let $G\in \generators _{W}$.  Then the map $G\rightarrow \ast$ is an isomorphism
		in $\homotcat (\mathbf L _{\generators _{W}}\mathcal A)$.
		\item \label{lem.leftmapstriang.c} Let $f$ be a map in $sW$ (see \ref{exam.invmaps}).  
		Then $f$ becomes an isomorphism in 
		$\homotcat (\mathbf L _{\generators _{W}}\mathcal A)$.
	\end{enumerate}
\end{lem}
\begin{proof}
 	\eqref{lem.leftmapstriang.a}: It follows using the same argument as in \ref{prop.leftloctriang}.
	
	\eqref{lem.leftmapstriang.d} By \ref{lem.leftmapstriang}\eqref{lem.leftmapstriang.a} it suffices to show that $G\rightarrow \ast$
	becomes an isomorphism in $\homotcat (\mathbf L _{W}\mathcal A)$.  By construction of $\generators _{W}$, 
	there exists a map $f:E\rightarrow F$ in $W$, such that $\xymatrix{E\ar[r]^-{f} &F\ar[r]&G}$ is a distinguished triangle in
	$\homotcat (\mathcal A )$.  Thus, by \ref{lem.leftmapstriang}\eqref{lem.leftmapstriang.a} it suffices to check
	that $f$ becomes an isomorphism in $\homotcat (\mathbf L _{W}\mathcal A)$.  This follows directly from the definition of 
	$\mathbf L _{W}\mathcal A$ and the universal property of left Bousfield localizations \cite[3.1.1.(1)(a) and 3.3.19.(1)]{MR1944041}.
	
	\eqref{lem.leftmapstriang.b}: This follows directly from the construction of $\mathbf L _{\generators _{W}}\mathcal A$ and
	\cite[3.1.1.(1)(a) and 3.3.19.(1)]{MR1944041}.
	
	\eqref{lem.leftmapstriang.c}: By \ref{prop.leftloctriang} we can assume that $f$ is in $W$ (see \ref{exam.invmaps}).  By construction
	there is a distinguished triangle in $\homotcat (\mathcal A)$ with $G\in \generators_{W}$:
			$\xymatrix{E\ar[r]^-{f}&F\ar[r]&G}$.
	Combining \ref{lem.leftmapstriang}\eqref{lem.leftmapstriang.b}  and \ref{prop.leftloctriang} we deduce that
	$f$ becomes an isomorphism in $\homotcat (\mathbf L_{\generators_{W}}\mathcal A)$.
\end{proof}

\begin{prop}  \label{prop.invmaps}
	The identity functor:
		\[	(id, id, \varphi): \mathbf L _{W}\mathcal A \rightarrow \mathbf L _{\generators _{W}}\mathcal A
		\]
	is a Quillen equivalence.
\end{prop}
\begin{proof}
	Since $\mathbf L_{W}\mathcal A$, $\mathbf L_{\generators_{W}}\mathcal A$ are left Bousfield localizations
	of $\mathcal A$, we deduce that they are simplicial model categories with the same cofibrant 
	replacement functor $Q$.  
	Thus, it suffices to show that they have the same class of
	weak equivalences.  By \cite[9.7.4]{MR1944041}, it is enough to check that the fibrant objects in
	$\mathbf L_{W}\mathcal A$ and $\mathbf L_{\generators_{W}}\mathcal A$ are the same.
	
	 For this, it suffices to show \cite[3.1.6.(c) and 4.1.1.(2)]{MR1944041} that the identity functor
	 $id: \mathbf L_{W}\mathcal A\rightarrow \mathbf L _{\generators_{W}}\mathcal A$ (resp.
	 $id: \mathbf L _{\generators_{W}}\mathcal A \rightarrow \mathbf L_{W}\mathcal A$) is  a left Quillen functor.  We
	consider the following diagram, where the arrows are left Quillen functors:
		\[  \xymatrix{\mathcal A \ar[r]^-{id} \ar[d]_-{id}& \mathbf L_{\generators_{W}}\mathcal A \\
						\mathbf L_{W}\mathcal A &}
		\]
	By \ref{lem.leftmapstriang}\eqref{lem.leftmapstriang.c} 
	(resp. \ref{lem.leftmapstriang}\eqref{lem.leftmapstriang.d}),
	the maps in $sW$  (resp. $M_{\generators _{W}}$)
	become isomorphisms in $\homotcat (\mathbf L_{\generators_{W}}\mathcal A)$ (resp. $\homotcat (\mathbf L_{W}\mathcal A)$).  Thus,
	the universal property of left Bousfield localizations  
	\cite[3.3.19.(1) and 3.1.1.(1)]{MR1944041} implies that $id: \mathbf L_{W}\mathcal A\rightarrow \mathbf L _{\generators_{W}}\mathcal A$ (resp.
	 $id: \mathbf L _{\generators_{W}}\mathcal A \rightarrow \mathbf L_{W}\mathcal A$) is a left Quillen functor.
\end{proof}

	The next two lemmas characterize the fibrant objects in $\mathbf L _{\generators}\mathcal A$.  Recall that $L$ is a fibrant
	replacement functor in $\mathbf L _{\generators}\mathcal A$.

\begin{lem}  \label{lem.leftobj.orth}
	Let $K$ be an arbitrary object in $\homotcat (\mathbf L _{\generators}\mathcal A)$.  Then $LK$ is in
	$Loc(\generators)^{\perp}$.
\end{lem}
\begin{proof}
	We will assume that the objects $\susp{n}G$ appearing as domains in
	$M_{\generators}$ (see \eqref{eq.stable.maps}) are cofibrant in $\mathcal A$ (see \ref{rmk.obj.cof}).  The universal property
	of left Bousfield localizations \cite[3.3.19.(1) and 3.1.1.(1)]{MR1944041}
	implies that for every $n\in \mathbb Z$ and every $G\in \generators$,
	the map $\susp{n}G\rightarrow \ast$ is an isomorphism in $\homotcat (\mathbf L _{\generators} \mathcal A)$.
	By adjointness: 
	\[\Hom _{\mathcal T}(\susp{n}G,LK)\cong \Hom _{\homotcat (\mathbf L _{\generators} \mathcal A)}(\susp{n}G,K)=0
	\]
	for every $n\in \mathbb Z$ and every $G\in \generators$.  Thus, we conclude that $\generators$ is contained in
	$^{\perp}\! LK$, where $^{\perp}\! LK$ is the full subcategory of $\mathcal T=\homotcat (\mathcal A)$ with objects $A$ such that
	for every $n\in \mathbb Z$, $\Hom _{\mathcal T}(\susp{n}A, LK)=0$.  
	
	We observe that $^{\perp}\! LK$ is a full triangulated subcategory of $\mathcal T$ (see \ref{lem.orttriang})
	and is closed under arbitrary coproducts.  
	Thus, $Loc(\generators)\subseteq 
	^{\perp}\! LK$, i.e. $LK\in Loc(\generators)^{\perp}$.
\end{proof}

\begin{lem}  \label{lem.orth.local}
	Let $K$ be a fibrant object in $\mathcal A$.  Assume that $K$ is in $Loc(\generators)^{\perp}$.  Then
	$K$ is also fibrant in $\mathbf L _{\generators} \mathcal A$.
\end{lem}
\begin{proof}
	We will assume that the objects $\susp{n}G$ appearing as domains in
	$M_{\generators}$ (see \eqref{eq.stable.maps}) are cofibrant in $\mathcal A$ (see \ref{rmk.obj.cof}).
	By \cite[3.1.4.(1)(a) and 4.1.1]{MR1944041} it suffices to show that the map:
		$ \ast=Map(\ast, K)\rightarrow Map(\susp{n}G,K),
		$
	is a weak equivalence of simplicial sets for every $n\in \mathbb Z$ and every $G\in \generators$.
	
	Let $\omega$ be a base point for $Map(\susp{n}G,K)$.
	Since $\mathcal A$ is a simplicial model category, we deduce $\pi _{0}Map(\susp{n}G,K)\cong 
	\Hom _{\mathcal T}(\susp{n}G,K)$ which is zero by hypothesis when $G\in \generators$.  Thus, we conclude that
	$Map(\susp{n}G,K)$ has only one connected component, and as a consequence it suffices to show that 
	for $r\geq 1$ the homotopy groups $\pi_{r,\omega} Map(\susp{n}G,K)=\ast$.
	
	By \cite[Lem. 6.1.2]{MR1860878}, we conclude $\pi_{r,\omega} Map(\susp{n}G,K) \cong \Hom _{\mathcal T}(\susp{n+r}G,K)$
	which vanishes by hypothesis.  Hence the result follows.
\end{proof} 

	The following theorem is the main result of this section.

\begin{thm}  \label{thm.presorth}
	Assume that $\generators$ is a set.  Then,  
	$Loc(\generators)^{\perp}$ is naturally equivalent as a triangulated category to the homotopy category 
	$\homotcat (\mathbf L _{\generators}\mathcal A)$.
\end{thm}
\begin{proof}
	Combining \ref{prop.leftloctriang} and \ref{lem.leftobj.orth}, we deduce that the triangulated functor
	$L: \homotcat (\mathbf L_{\generators} \mathcal A)\rightarrow \mathcal T =\homotcat (\mathcal A)$, factors through
	the inclusion $j:Loc(\generators)^{\perp} \rightarrow \mathcal T$,
	which is also a triangulated functor.  
	
	We will abuse notation and write 
	$L: \homotcat (\mathbf L_{\generators} \mathcal A)\rightarrow Loc(\generators) ^{\perp}$ for the corresponding
	factorization.  By \cite[3.5.2(1)]{MR1944041}, we conclude that $L$ is fully faithful.  Finally, lemma \ref{lem.orth.local}
	implies that $L$ is essentially surjective on objects, and so an equivalence of triangulated categories.
\end{proof}	

\begin{cor}  \label{cor.r-orth.loc}
	Assume that $\generators$ is a set.  Then,  
	$Loc(\generators)^{\perp}$ is a localizing subcategory of $\mathcal T =\homotcat (\mathcal A)$.
\end{cor}
\begin{proof}
	Recall (see \ref{subsec.Quillenmodcats}) that $\mathcal A$ is in particular closed under arbitrary coproducts.  Thus, by
	\ref{thm.presorth} we deduce that $Loc(\generators)^{\perp}$ is closed under arbitrary coproducts.  Therefore, the result
	follows from \ref{lem.orttriang}.
\end{proof}

	We are particularly interested in conditions that guarantee that  $Loc(\generators)^{\perp}$ is a compactly generated
	triangulated category in the sense of Neeman \cite[Def. 1.6]{MR1308405}.
	
	Recall that $L$ is a functorial fibrant replacement in $\mathbf L _{\generators}\mathcal A$.
	
\begin{thm} \label{thm.perpcompgen}
	Assume that $\generators$ is a set and
	that $\mathcal T = \homotcat (\mathcal A)$ is a compactly generated
	triangulated category in the sense of Neeman \cite[Def. 1.6]{MR1308405}, with 
	set of compact generators $\generatorstot$ \cite[Def. 1.8]{MR1308405}.
	Then,  $Loc(\generators)^{\perp}$ is a compactly generated
	triangulated category in the sense of Neeman, with  set of compact generators:
	\[  L\generatorstot = \{LK: K\in \generatorstot \}.
	\]
\end{thm}
\begin{proof}
	By \ref{thm.presorth}, it suffices to show that $\homotcat (\mathbf L_{\generators} \mathcal A)$ is compactly generated
	with generators $\{LK: K\in \generatorstot \}$.  
	
	Let $B\in \homotcat (\mathbf L_{\generators} \mathcal A)$ such that
	for every $K\in \generatorstot$: $\Hom _{\homotcat (\mathbf L_{\generators} \mathcal A)}(LK,B)=0$.  Then $B\cong \ast$ in
	$\homotcat (\mathbf L_{\generators} \mathcal A)$.  In effect, we may assume that $B$ is fibrant in 
	$\mathbf L_{\generators} \mathcal A$, then by adjointness:
	\[	\Hom _{\mathcal T}(K,B) \cong \Hom _{\homotcat (\mathbf L_{\generators} \mathcal A)}(K, B) =
	\Hom _{\homotcat (\mathbf L_{\generators} \mathcal A)}(LK, B)=0			
	\]
	Since $\generatorstot$ is a set of compact generators for $\mathcal T$, we deduce that $B\cong \ast$ in $\mathcal T$.
	By \cite[3.3.3.(1)(a)]{MR1944041} we conclude that $B\cong \ast$ in $\homotcat (\mathbf L_{\generators} \mathcal A)$.
	
	It only remains to show that the objects in $L\generatorstot = \{ LK:K\in \generatorstot \}$ are compact in 
	$\homotcat (\mathbf L_{\generators} \mathcal A)$.
	Let $B_{\lambda}$ be a family of objects in $\homotcat (\mathbf L_{\generators} \mathcal A)$ 
	indexed by a set $\Lambda$.  We need to show:
	\[	\Hom _{\homotcat (\mathbf L_{\generators} \mathcal A)}(LK, \oplus _{\lambda \in \Lambda}B_{\lambda})
			 \cong \oplus _{\lambda \in \Lambda} \Hom _{\homotcat (\mathbf L_{\generators} \mathcal A)} (LK, B_{\lambda})
	\]
	We may assume that $\oplus _{\lambda \in \Lambda}B_{\lambda}$ is fibrant in 
	$\mathbf L_{\generators} \mathcal A$.  By \ref{lem.orth.local}, we deduce that for every $\lambda \in \Lambda$,
	$B_{\lambda}$ is also fibrant in $\mathbf L_{\generators} \mathcal A$.  By adjointness and  the compactness of $K$
	in $\mathcal T$:
	\begin{align*}	
			\Hom _{\homotcat (\mathbf L_{\generators} \mathcal A)}(LK, \oplus _{\lambda \in \Lambda}B_{\lambda})
			& = \Hom _{\homotcat (\mathbf L_{\generators} \mathcal A)}(K, \oplus _{\lambda \in \Lambda}B_{\lambda})
			\cong \Hom _{\mathcal T}(K,  \oplus _{\lambda \in \Lambda}B_{\lambda})\\ &\cong
			\oplus _{\lambda \in \Lambda} \Hom _{\mathcal T} (K, B_{\lambda}) \cong
			\oplus _{\lambda \in \Lambda} \Hom _{\homotcat (\mathbf L_{\generators} \mathcal A)} (K, B_{\lambda}) \\
			& = \oplus _{\lambda \in \Lambda} \Hom _{\homotcat (\mathbf L_{\generators} \mathcal A)} (LK, B_{\lambda})
	\end{align*}
	Hence the result follows.
\end{proof}

	We obtain the following interesting corollaries, where
	$L$ is a functorial fibrant replacement in $\mathbf L _{\generators}\mathcal A$.
	
\begin{cor}  \label{cor.perp.loc}
	Assume that $\generators$ is a set and
	that $\mathcal T = \homotcat (\mathcal A)$ is a compactly generated
	triangulated category in the sense of Neeman \cite[Def. 1.6]{MR1308405}, with 
	set of compact generators $\generatorstot$ \cite[Def. 1.8]{MR1308405}.
	Then,  $Loc(\generators)^{\perp}=Loc(L\generatorstot)$ (see \ref{eq.localizing.gen}).
\end{cor}
\begin{proof}
	It follows from \ref{thm.perpcompgen}  and \cite[Thm. 2.1(2.1.2)]{MR1308405}.
\end{proof}	
	
	With the notation of \ref{subsubsec.presloc}, \ref{eq.stable.generators}
	let $\mathbf{R}_{L\generatorstot}\mathcal A$ be the right Bousfield localization  of
	 $\mathcal A$ with respect to the set:
		$\{ \susp{n}LK: K\in \generatorstot , n\in \mathbb Z \}$.

\begin{cor} \label{cor.perp.rightbous}
	Assume that $\generators$ is a set and
	that $\mathcal T = \homotcat (\mathcal A)$ is a compactly generated
	triangulated category in the sense of Neeman \cite[Def. 1.6]{MR1308405}, with 
	set of compact generators $\generatorstot$ \cite[Def. 1.8]{MR1308405}.
	Then,  $Loc(\generators)^{\perp}$ is naturally equivalent as a triangulated category to the homotopy category 
	$\homotcat (\mathbf R _{L\generatorstot}\mathcal A)$.
\end{cor}	
\begin{proof}
	By \ref{cor.perp.loc},  $Loc(\generators)^{\perp}=Loc(L\generatorstot)$.  
	Hence, the result follows from \ref{thm.perpcompgen} and \ref{thm.preslocsub}.
\end{proof}

	Combining \ref{thm.presorth} and \ref{cor.perp.rightbous} we obtain two different presentations for $Loc(\generators)^{\perp}$, 
	one as a left Bousfield localization and the other as a right Bousfield localization:
	
\begin{cor}  \label{cor.perp.2prestns}
	Assume that $\generators$ is a set and
	that $\mathcal T = \homotcat (\mathcal A)$ is a compactly generated
	triangulated category in the sense of Neeman \cite[Def. 1.6]{MR1308405}, with 
	set of compact generators $\generatorstot$ \cite[Def. 1.8]{MR1308405}.
	Then,  
		\[	Loc(\generators)^{\perp} \cong \homotcat (\mathbf L _{\generators}\mathcal A) \cong
			\homotcat (\mathbf R _{L\generatorstot}\mathcal A)
		\]
	are naturally equivalent as triangulated categories. 
\end{cor}

\section{Weakly Birational Coverings and the Weakly  Birational Tower}  \label{sec.birational}

In \cite{MR3035769} we introduced the birational (resp. weakly birational) motivic stable homotopy
categories.  In this section, after recalling their definition we show that we can apply to them the formalism
of \S 2 and then use this formalism to construct the (weakly) birational coverings 
and the birational tower in the Morel-Voevodsky
motivic stable homotopy category.

Recall that $\Tspectra$ is Jardine's category of symmetric $T$-spectra on $\unsmot$ 
equipped with the motivic model structure  
\cite[theorem 4.15]{MR1787949}.  We will write $\stablehomotopy$ for its homotopy category, which is triangulated.
		
\subsection{Birational motivic stable homotopy categories}

	We will use Jardine's notation \cite[p. 506-507]{MR1787949}.
	Namely, let $F_{n}$ denote the left adjoint to the $n$-evaluation functor:
		\[ \xymatrix@R=0.5pt{\Tspectra \ar[r]^-{ev_{n}}& \mathcal M \\
					(E^{m})_{m\geq 0} \ar@{|->}[r]& E^{n}}
		\]
	Notice that $F_{0}(A)$ is just the usual infinite suspension spectrum $\Sigma _{T}^{\infty}A$.
	The sphere spectrum $\sphere$ is $F_{0}(S^{0})$.
	
	\subsubsection{Generators for $\stablehomotopy$}  \label{subsubsec.genSH}
	
	Recall that $\stablehomotopy$ is a compactly generated triangulated category in the sense of Neeman \cite[Def. 1.7]{MR1308405}
	with set of compact generators \cite[Prop. 3.1.5]{MR2807904}:
		\begin{align}  \label{gens.SH}
			\generators _{\stablehomotopy}=\{ F_{n}(S^{p}\wedge \Gm^{q}\wedge U_{+}): U\in Sm_{X}; n,p,q\geq 0 \}.
		\end{align}  	
	where $U_{+}$ denotes the simplicial presheaf represented by $U$ with a disjoint base point.
	
\begin{rmk}  \label{desusp.sh}
	The desuspension functor $\susp{-1}:\stablehomotopy \rightarrow \stablehomotopy$ can be written in this notation as
	$E\mapsto F_{1}(\Gm \wedge E)$.  If $q>0$, we will write $\Gm ^{-q}\wedge E$ for $F_{q}(S^{q}\wedge E)$.
\end{rmk}
		
\subsubsection{Codimension} \label{def.codimension}
	Let $Y\in Sch_{X}$, and $Z$ a closed subscheme of $Y$.  The \emph{codimension}
	of $Z$ in $Y$ (see \cite[section 7.5]{MR0338129}), $codim_{Y}Z$ is the infimum (over the generic points $z_{i}$ of $Z$)
	of the dimensions of the local rings $\mathcal O _{Y, z_{i}}$. 

	Since $X$ is Noetherian of finite Krull dimension and $Y$ is of finite type over $X$, $codim_{Y}Z$ is always finite.
		
\subsubsection{Open immersions}  \label{def.localizing-maps}
	Let  $n\geq 0$ be an integer, and consider the following set of 
	open immersions which have a closed complement of codimension at least $n+1$
\begin{align*}
	B_{n}=\{ \iota _{U,Y}:& U\rightarrow Y \text{ open immersion } |\\														
											 & Y\in Sm_{X}; Y \text{ irreducible};
									     (codim_{Y}Y\backslash U)\geq n+1
	\}
\end{align*}
	
\subsubsection{Open immersions with smooth complement} \label{def.localizing-maps2}
	Let  $n\geq 0$, and consider the following set of 
	open immersions with smooth closed complement	of codimension at least $n+1$
\begin{align*}
	WB_{n}=\{ \iota _{U,Y}: &U\rightarrow Y \text{ open immersion } |\\
									& Y, Z=Y\backslash U \in Sm_{X}; Y \text{ irreducible};
									      (codim_{Y}Z)\geq n+1
	\}
\end{align*}	
	Now we will study the left Bousfield localizations of $\Tspectra$ with respect to a suitable set of
	maps induced by the families of open immersions $B_{n}$, $WB_{n}$ described above.
	
\subsubsection{First presentation} \label{def.localizationAmod-Birat}
	Let $n\in \mathbb Z$ be an arbitrary integer.
	\begin{enumerate}
		\item \label{def.localizationAmod-Birat.a}  
		We will write $\Tspectranbirat{n}$ (resp. $\Tspectranwbirat{n}$) for the
				left Bousfield localization of $\Tspectra$ 
				with respect to the set of maps:
					\begin{align*}	sB_{n} & =\{ F_{p}(\Gm ^{b}\wedge \iota _{U,Y}): b, p, r\geq 0, 
							b-p\geq n-r; \iota _{U,Y}\in B_{r} \} \\
							(\text{resp. } sWB_{n} & =\{ F_{p}(\Gm ^{b}\wedge \iota _{U,Y}): b, p, r\geq 0, 
								b-p\geq n-r; \iota _{U,Y}\in WB_{r} \}).
					\end{align*}
		\item	\label{def.localizationAmod-Birat.b}  
		We will write $b^{(n)}$ (resp. $wb^{(n)}$) for a fibrant replacement functor  in $\Tspectranbirat{n}$ (resp. $\Tspectranwbirat{n}$)
					and $\stablehomotopynbirat{n}$ (resp. $\stablehomotopynwbirat{n}$) for its associated homotopy category. 
	\end{enumerate}
	For $n\neq 0$ we will call $\stablehomotopynbirat{n}$ (resp. $\stablehomotopynwbirat{n}$) 
	the codimension $(n+1)$-birational motivic stable homotopy
	category (resp. codimension $(n+1)$-weakly birational motivic stable homotopy category), 
	and for $n=0$ we will call it the birational motivic stable homotopy category
	(resp. weakly birational motivic stable homotopy category).

\subsubsection{Second presentation}  \label{def.locbirat-gens}
	\begin{enumerate}
		\item	\label{def.locbirat-gens.a}  Let $\generators _{B_{n}}$ (resp. $\generators _{WB_{n}}$) 
				be the set of objects in $\Tspectra$
				which complete the maps in $sB_{n}$ (resp. $sWB_{n}$) to cofibre sequences in $\Tspectra$, i.e. distinguished triangles
				in $\stablehomotopy$.  Namely, there is a bijection of sets 
				$sB_{n}\rightarrow \generators _{B_{n}}$ (resp. $sWB_{n}\rightarrow \generators _{WB_{n}}$), $f\mapsto G_{f}$
				such that for every $f:E\rightarrow F$ in $sB_{n}$ (resp. $sWB_{n}$), 
				$G_{f}$ fits in a cofibre sequence in $\Tspectra$: $E \rightarrow F \rightarrow G_{f}$.
		\item \label{def.locbirat-gens.b} Let $\mathbf L _{\generators _{B_{n}}}\Tspectra$ 
				(resp. $\mathbf L _{\generators _{WB_{n}}}\Tspectra$) be the left Bousfield localization
				of $\Tspectra$ with respect to the set of maps $M_{\generators _{B_{n}}}$ (resp. $M_{\generators _{WB_{n}}}$),
				see \ref{eq.stable.maps}.
		\item \label{def.locbirat-gens.c} We will write $Loc(\generators _{B_{n}})$ (resp. $Loc(\generators _{WB_{n}})$) 
				for the localizing subcategory of $\stablehomotopy$ generated
				by $\generators _{B_{n}}$ (resp. $\generators _{WB_{n}}$).
		\item \label{def.locbirat-gens.d} We will write $B_{n}^{\perp}$ (resp. $WB_{n}^{\perp}$) for 
				$Loc(\generators _{B_{n}})^{\perp}$ (resp. $Loc(\generators _{WB_{n}})^{\perp}$) in $\stablehomotopy$.
	\end{enumerate}

\begin{thm}  \label{thm.newpres.birat.a}
	The identity functor:
		\begin{align*}
			id & : \Tspectranbirat{n} \rightarrow \mathbf L _{\generators _{B_{n}}}\Tspectra \\
			(\text{resp. } id & : \Tspectranwbirat{n} \rightarrow \mathbf L _{\generators _{WB_{n}}}\Tspectra).
		\end{align*} 
	is a Quillen equivalence.
\end{thm}
\begin{proof}
	We just need to consider the first claim,
	the proof being exactly the same in the second case.
	By \ref{exam.invmaps} and \ref{prop.invmaps}, 
	it suffices to check that $\Tspectranbirat{n}$ (see \ref{def.localizationAmod-Birat}) is also 
	the left Bousfield localization of $\Tspectra$ with respect to the set of maps $\{ \susp{r}f: f\in sB_{n}; r\geq 0\}$.  
	This follows directly from
	Lemma 2.5 in \cite{MR3035769} and the universal property of left Bousfield localizations  \cite[3.1.1.(1)(a) and 3.3.19.(1)]{MR1944041}.
\end{proof}

\subsubsection{Third presentation}  \label{def.rightpres.birat}	

	Recall that $b^{(n)}$ (resp. $wb^{(n)}$) is a fibrant replacement functor in
	$\Tspectranbirat{n}$ (resp. $\Tspectranwbirat{n}$).  By \ref{thm.newpres.birat.a}, $b^{(n)}$ (resp. $wb^{(n)}$)
	is also a fibrant replacement functor in $\mathbf L _{\generators _{B_{n}}}\Tspectra$ 
	(resp. $\mathbf L _{\generators _{WB_{n}}}\Tspectra$).
	We will write $\mathbf{R}_{b^{(n)}\generators _{\stablehomotopy}}\Tspectra$ 
	(resp. $\mathbf{R}_{wb^{(n)}\generators _{\stablehomotopy}}\Tspectra$) for the right Bousfield localization  of
	 $\Tspectra$ with respect to the set (see \ref{gens.SH} and \ref{subsubsec.presloc}, \ref{eq.stable.generators}):
	 \begin{align*} 
		\{ \susp{m}(b^{(n)}G) & : G\in \generators _{\stablehomotopy} , m\in \mathbb Z \} \\
		(\text{resp. } \{ \susp{m}(wb^{(n)}G) & : G\in \generators _{\stablehomotopy} , m\in \mathbb Z \}).
	\end{align*}

Recall that $B_{n}^{\perp}$, $WB_{n}^{\perp}$ is respectively $Loc(\generators _{B_{n}})^{\perp}$
and $Loc(\generators _{WB_{n}})^{\perp}$.

\begin{thm}  \label{thm.newpres.birat.b}
	The following:
	\begin{align*}	
			\stablehomotopynbirat{n} & \cong B_{n}^{\perp}  \cong 
			\homotcat (\mathbf{R}_{b^{(n)}\generators _{\stablehomotopy}}\Tspectra)\\
			(\text{resp. } \stablehomotopynwbirat{n} & \cong WB_{n}^{\perp}  \cong 
			\homotcat (\mathbf{R}_{wb^{(n)}\generators _{\stablehomotopy}}\Tspectra) ).
	\end{align*}
	are naturally equivalent as triangulated categories.
\end{thm}
\begin{proof}
	We just need to consider the first claim,
	the proof being exactly the same in the second case.
	By \ref{subsubsec.genSH}, $\stablehomotopy$ is compactly generated 
	with set of compact generators $\generators _{SH}$.  Hence the result follows by
	combining \ref{thm.newpres.birat.a} and \ref{cor.perp.2prestns}.
\end{proof}
	
	Recall that $\generators _{\stablehomotopy}$  is a set of compact generators for $\stablehomotopy$
	(see \ref{gens.SH}).
	
\begin{cor}
		\label{cor.ortcp}
	The inclusion (see \ref{def.locbirat-gens}\eqref{def.locbirat-gens.d}):
	\begin{align*} 
		j_{n} & : B_{n}^{\perp}\rightarrow \stablehomotopy \\
	(\text{resp. } j_{n} & : WB_{n}^{\perp}\rightarrow \stablehomotopy)
	\end{align*}
	is a triangulated functor which commutes
	with arbitrary coproducts.  In addition,
	 $B_{n}^{\perp}$ (resp. $WB_{n}^{\perp}$) is a compactly generated triangulated category
	 in the sense of Neeman \cite[Def. 1.7]{MR1308405} with set of compact generators:
	 	\begin{align*}	
			b^{(n)}\generators _{\stablehomotopy} & =\{ \susp{m}(b^{(n)}G): 
			G\in \generators _{\stablehomotopy} , m\in \mathbb Z \}\\
			(\text{resp. } wb^{(n)}\generators _{\stablehomotopy} & =\{ \susp{m}(wb^{(n)}G): G\in 
			\generators _{\stablehomotopy} , m\in \mathbb Z \}).
		\end{align*}
\end{cor}
\begin{proof}
	We just need to consider the first claim,
	the proof being exactly the same in the second case.
	It is clear that the inclusion $B_{n}^{\perp}\rightarrow \stablehomotopy$ is a triangulated functor and that
	it respects arbitrary coproducts.  By \ref{subsubsec.genSH}, $\stablehomotopy$ is compactly generated with set of compact generators $\generators _{SH}$.
	The compactness of $B_{n}^{\perp}$ follows from
	\ref{thm.newpres.birat.b} and \ref{thm.perpcompgen}.
\end{proof}

\subsubsection{Voevodsky's effective categories} \label{Vov.eff}

For every integer $q\in \mathbb Z$, consider the following family of symmetric $T$-spectra:
		\begin{align}
				\label{eqn.Cqeff}
		 	C^{q}_{\mathit{eff}}=\{ F_{n}(S^{r}\wedge \mathbb G _{m}^{s}\wedge U_{+}) 
				\mid n,r,s \geq 0; s-n\geq q; U\in Sm_{X}\}
		\end{align}
	Voevodsky \cite{MR1977582} defines the $q$-\emph{effective} motivic stable homotopy category 
	$\neffstablehomotopy{q}$ to be $Loc(C^{q}_{\mathit{eff}})$ (see \ref{eq.localizing.gen}), i.e.
	$\neffstablehomotopy{q}$ is the smallest full triangulated subcategory of $\stablehomotopy$
	which contains $C^{q}_{\mathit{eff}}$ and is closed under arbitrary coproducts.  When $q=0$ we will
	simply write $\stablehomotopyeff$ for $\neffstablehomotopy{0}$.
	
\begin{defi}
		\label{def.orth}
	Let $E\in \Tspectra$ be a symmetric $T$-spectrum.		
	We will say that $E$ is \emph{$n$-orthogonal},
	if for all $K\in \neffstablehomotopy{n}$:
		$\Hom _{\stablehomotopy}(K,E)=0$.
	Let $\northogonal{n}$ denote the full subcategory of $\stablehomotopy$ consisting
	of the $n$-orthogonal objects.
\end{defi}

\begin{thm}
		\label{thm.ort=wbirationalTspectra}
	Let $n$ be an arbitrary integer.  Then
	the full triangulated subcategories of $\stablehomotopy$: $WB_{n}^{\perp}$ and
	$\northogonal{n+1}$ are identical (see \ref{def.locbirat-gens}\eqref{def.locbirat-gens.d}).
	In addition, if the base scheme $X=\spec{k}$, with $k$ a perfect field,
	then the full triangulated subcategories of $\stablehomotopy$: $B_{n}^{\perp}$ and
	$\northogonal{n+1}$ are identical.
\end{thm}
\begin{proof}
	By \ref{thm.newpres.birat.a} and
	\ref{thm.presorth}, $WB_{n}^{\perp}$ is naturally isomorphic as triangulated category to
	$\homotcat (\Tspectranwbirat{n})=\stablehomotopynwbirat{n}$.  On the other hand, by
	\cite[Thms. 1.4(2) and 3.6]{MR3035769}, $\northogonal{n+1}$ is also naturally isomorphic as triangulated category to
	$\homotcat (\Tspectranwbirat{n})=\stablehomotopynwbirat{n}$.
	
	Then, the result follows combining \ref{lem.leftobj.orth}, \ref{lem.orth.local} and \cite[Prop. 3.5]{MR3035769}.  
	The case of a perfect field follows from Prop. 2.12 in \cite{MR3035769}.
\end{proof}

Recall that for $q>0$, $\Gm^{-q}\wedge E$ is $F_{q}(S^{q}\wedge E)$ (see \ref{desusp.sh}).

\begin{cor}  \label{lem.Gmshift.orth}
	Let $E$ be an arbitrary symmetric $T$-spectrum in $\stablehomotopy$, and
	$n, q$ be arbitrary integers.  Then
	$E$ is in $WB_{n}^{\perp}$ if and only if $\Gm ^{q}\wedge E$ is in $WB_{n+q}^{\perp}$.
\end{cor}
\begin{proof}
	By \ref{thm.ort=wbirationalTspectra}, it follows that $WB_{n}^{\perp}$ is equal to 
	$\northogonal{n+1}$, which by definition is $(\neffstablehomotopy{n+1})^{\perp}$ (see \ref{def.orth}).  To conclude,
	we observe that the triangulated functor
	$\Gm ^{q}:\stablehomotopy \rightarrow \stablehomotopy$, $E\mapsto \Gm^{q} \wedge E$;
	admits an inverse ($E\mapsto \Gm^{-q}\wedge E$) and maps $\neffstablehomotopy{n+1}$ surjectively onto $\neffstablehomotopy{n+q+1}$.
\end{proof}

\subsection{The weakly birational tower}  \label{subsec.birattow}  

In this section $q$ will be an arbitrary integer.
By construction there is an inclusion $sWB_{q+1}\subseteq sWB_{q}$ of sets of maps in $\Tspectra$
(see \ref{def.localizing-maps}, \ref{def.localizing-maps2}, 
\ref{def.localizationAmod-Birat}\eqref{def.localizationAmod-Birat.a}),  and hence an inclusion 
$\generators_{WB_{q+1}} \subseteq \generators_{WB_{q}}$  of sets of objects in $\Tspectra$
(see \ref{def.locbirat-gens}\eqref{def.locbirat-gens.a}).

Thus, we deduce that
$Loc(\generators _{WB_{q+1}})\subseteq Loc(\generators _{WB_{q}})$; and as a consequence 
(see \ref{def.locbirat-gens}\eqref{def.locbirat-gens.d}):
\[WB_{q}^{\perp}\subseteq WB_{q+1}^{\perp}.\]  
which are localizing subcategories of $\stablehomotopy$ (see \ref{cor.r-orth.loc} and \ref{def.locbirat-gens}\eqref{def.locbirat-gens.d}).
Therefore,
we obtain the following tower of localizing subcategories in $\stablehomotopy$:
		\begin{align}  \label{eq.birtower}
			\cdots \subseteq WB_{q-1}^{\perp}
			\subseteq WB_{q}^{\perp}
			\subseteq WB_{q+1}^{\perp} \subseteq \cdots
		\end{align}	

\begin{prop}
		\label{prop.adj-ort.b}
	The inclusion,
		$ j_{q}: WB_{q}^{\perp} \rightarrow \stablehomotopy$
	admits a right adjoint:
		\[p_{q}:\stablehomotopy \rightarrow WB_{q}^{\perp},\]
	which is also a triangulated functor.
\end{prop}
\begin{proof}
	The result follows by combining \ref{cor.ortcp} and
	theorem 4.1 in \cite{MR1308405}.
\end{proof}

\begin{rmk}
		\label{rmk.3.1.unit=iso}
	Since the inclusion $j_{q}: WB_{q}^{\perp} \rightarrow \stablehomotopy$
					is a full embedding,
					we deduce that the unit of the adjunction $id\stackrel{\tau}{\rightarrow} p_{q}j_{q}$
					is a natural isomorphism.
\end{rmk}
	
\subsubsection{Weakly birational cover} \label{def.birat.cover}		
		We define $wb^{c}_{q}=j_{q}p_{q}$.  Then clearly $wb^{c}_{q} \circ wb^{c}_{q+1}=wb^{c}_{q}$ and
					there exists a canonical natural transformation $wb^{c}_{q}\rightarrow wb^{c}_{q+1}$.
	
The following proposition is well-known.

\begin{prop}
		\label{prop.3.1.counit-properties}
	The counit $wb^{c}_{q}=j_{q}p_{q}\stackrel{\theta_{q}}{\rightarrow} id$ of the adjunction constructed in \ref{prop.adj-ort.b}
	satisfies the following universal property: 
	
	For any symmetric $T$-spectrum $E$ in $\stablehomotopy$ and for any $F\in WB_{q}^{\perp}$,
	the map $\theta ^{E}_{q}: wb^{c}_{q}E \rightarrow E$
	in $\stablehomotopy$ induces an isomorphism of
	abelian groups:
		\[\xymatrix{\Hom _{\stablehomotopy}(F, wb^{c}_{q}E) \ar[r]_-{\cong}^-{\theta ^{E}_{q\ast}}& 
			\Hom _{\stablehomotopy}(F, E) }
		\]
\end{prop}
\begin{proof}
	If $F\in WB_{q}^{\perp}$, then $\Hom _{\stablehomotopy}(F, E)=\Hom _{\stablehomotopy}(j_{q}F, E)$.
	By adjointness:
		\[  \Hom _{\stablehomotopy}(j_{q}F, E)\cong \Hom _{WB_{q}^{\perp}}(F,p_{q}E).
		\]
	Since $WB_{q}^{\perp}$ is a full subcategory of $\stablehomotopy$, we deduce that:
	\[ \Hom _{WB_{q}^{\perp}}(F,p_{q}E)=
	\Hom _{\stablehomotopy}(j_{q}F, j_{q}p_{q}E)=
	\Hom _{\stablehomotopy}(F, wb^{c}_{q}E).
	\]  
	This finishes the proof.
\end{proof}

\begin{cor}  \label{cor.comp.tool1}
	Let $E$ be an arbitrary symmetric $T$-spectrum in $\stablehomotopy$.  Then, the natural map $\theta ^{E}_{q}: wb^{c}_{q}E\rightarrow E$
	is an isomorphism in $\stablehomotopy$ if and only if $E$ belongs to $WB_{q}^{\perp}$.
\end{cor}
\begin{proof}
	By \ref{cor.r-orth.loc}, we conclude that
	$WB_{q}^{\perp}$ is a localizing subcategory of $\stablehomotopy$.
	
	First, we assume that $\theta ^{E}_{q}$ is an isomorphism in $\stablehomotopy$.  Then,
	by construction 
	$wb^{c}_{q}E$ is in $WB_{q}^{\perp}$ (see \ref{prop.adj-ort.b} and \ref{def.birat.cover}).   
	Hence, we deduce that $E$ is also in $WB_{q}^{\perp}$ since
	it is a localizing subcategory of $\stablehomotopy$.
	
	Finally, we assume that $E$ is in $WB_{q}^{\perp}$ which is a localizing subcategory of $\stablehomotopy$.  Thus,
	$\theta ^{E}_{q}$ is a map in $WB_{q}^{\perp}$ since $wb^{c}_{q}E$ is in
	$WB_{q}^{\perp}$ by construction.  Therefore, by \ref{prop.3.1.counit-properties} we deduce that
	$\theta ^{E}_{q}$ is an isomorphism in $WB_{q}^{\perp}$, and hence an isomorphism in $\stablehomotopy$.
\end{proof}

\begin{thm}
		\label{thm.3.1.slicefiltration}
	There exist triangulated functors:
		\[  \xymatrix{wb_{q+1/q}:\stablehomotopy \ar[r]& \stablehomotopy}
		\]
	together with natural transformations:
		\[ \xymatrix@R=.6pt{\pi _{q+1}:wb^{c}_{q+1} \ar[r]& wb_{q+1/q} \\
								\sigma _{q+1}:wb_{q+1/q} \ar[r]& \susp{1}wb^{c}_{q}}
		\]
	such that for any symmetric $T$-spectrum $E$ in $\stablehomotopy$ the following conditions hold:
	\begin{enumerate}
		\item	\label{thm.3.1.slicefiltration.a} There is a
					distinguished triangle in $\stablehomotopy$:
					\begin{equation}
						\label{equation.3.1.slicefitration}
							\xymatrix{wb^{c}_{q}E \ar[r]& wb^{c}_{q+1}E \ar[r]^-{\pi_{q+1}}& 
							 wb_{q+1/q}E \ar[r]^-{\sigma_{q+1}}& \susp{1}wb^{c}_{q}E}
					\end{equation}
		\item	\label{thm.3.1.slicefiltration.b} $wb_{q+1/q}E$ is in $WB_{q+1}^{\perp}$.
		\item	\label{thm.3.1.slicefiltration.c} $wb_{q+1/q}E$ is in $(WB_{q}^{\perp})^{\perp}$.  Namely,
					for any  $F$ in $WB_{q}^{\perp}$:
					\[ \Hom _{\stablehomotopy}(F, wb_{q+1/q}E)=0.\]  
	\end{enumerate}
\end{thm}
\begin{proof}
	By \ref{cor.ortcp},
	the triangulated categories $WB_{q}^{\perp}$ and $WB_{q+1}^{\perp}$
	are compactly generated.  Thus,  
	the result follows from 
	propositions 9.1.19 and 9.1.8 in \cite{MR1812507}.	
\end{proof}

\begin{thm}
		\label{thm.3.1.motivictower}
	There exist triangulated functors:
		\[ \xymatrix{wb_{>q}:\stablehomotopy \ar[r]& \stablehomotopy}
		\]
	together with natural transformations:
		\[ \xymatrix@R=.6pt{\pi _{>q}:id \ar[r]& wb_{>q} \\
								\sigma _{>q}:wb_{>q} \ar[r]& \susp{1}wb^{c}_{q}}
		\]
	such that for any symmetric $T$-spectrum $E$ in $\stablehomotopy$ the following conditions hold:
	\begin{enumerate}
		\item	\label{thm.3.1.motivictower.a} There is a
					distinguished triangle in $\stablehomotopy$:
					\begin{equation}
						\label{equation.3.1.motivictower}
							\xymatrix{wb^{c}_{q}E \ar[r]& E \ar[r]^-{\pi_{>q}}& 
							 wb_{>q}E \ar[r]^-{\sigma_{>q}}& \susp{1}wb^{c}_{q}E}
					\end{equation}
		\item	\label{thm.3.1.motivictower.b} $wb_{>q}E$ is in
					$(WB_{q}^{\perp})^{\perp}$.  Namely,  for any $F$ in $WB_{q}^{\perp}$:
					\[\Hom _{\stablehomotopy}(F,wb_{>q}E)=0.\]  
	\end{enumerate}
\end{thm}
\begin{proof}
	By \ref{subsubsec.genSH} and \ref{cor.ortcp},
	the triangulated categories $\stablehomotopy$ and $WB_{q}^{\perp}$ 
	are compactly generated.  Thus,  
	the result follows from 
	propositions 9.1.19 and 9.1.8 in \cite{MR1812507}.	
\end{proof}

\begin{thm}
		\label{thm.3.1.octahedral-axiom}
	For any symmetric $T$-spectrum $E$ in $\stablehomotopy$,
	there exists the following commutative diagram in $\stablehomotopy$: 
		\[ \xymatrix{wb^{c}_{q}E \ar[rr] \ar@{=}[d] && wb^{c}_{q+1}E \ar[rr]^-{\pi _{q+1}} \ar[d]&& wb_{q+1/q}E 
							\ar[rr]^-{\sigma _{q+1}} \ar[d]&& 
								\susp{1}wb^{c}_{q}E \ar@{=}[d]\\
								wb^{c}_{q}E \ar[rr] \ar[d]&& E \ar[rr]^-{\pi _{>q}} \ar[d]_-{\pi _{>q+1}}&& wb_{>q}E \ar[rr]^-{\sigma _{>q}} 
								\ar[d]&& \susp{1}wb^{c}_{q}E \ar[d]\\
								\ast \ar[rr] \ar[d]&& wb_{>q+1}E \ar@{=}[rr] \ar[d]_-{\sigma _{>q+1}}&& 
								wb_{>q+1}E \ar[rr] \ar[d]&& \ast \ar[d]\\
								\susp{1}wb^{c}_{q}E \ar[rr]&& \susp{1}wb^{c}_{q+1}E \ar[rr]_-{\susp{1}\pi _{q+1}}&&
								\susp{1}wb_{q+1/q}E \ar[rr]_-{\susp{1}\sigma _{q+1}}&& \susp{2}wb^{c}_{q}E}
		\]
	where all the rows and columns are distinguished triangles in $\stablehomotopy$.
\end{thm}
\begin{proof}
	The result follows from  \ref{thm.3.1.slicefiltration}, \ref{thm.3.1.motivictower} and
	the octahedral axiom applied to the following commutative diagram:
		\[  \xymatrix{wb^{c}_{q}E \ar[rr] \ar[dr]&& wb^{c}_{q+1}E \ar[dl]\\
								& E &}
		\]
\end{proof}

\subsubsection{The spectral sequence}

By \ref{thm.3.1.octahedral-axiom}, for every symmetric $T$-spectrum $E$ in $\stablehomotopy$ there is a tower in $\stablehomotopy$:

\begin{align}	\label{F.birtow}
	\xymatrix@C=1.5pc{\cdots \ar[r]
					& \ar[dr]|{\theta _{-1}^{E}} wb^{c}_{-1}(E) 
					\ar[r] & \ar[d]|{\theta _{0}^{E}} wb^{c}_{0}(E) \ar[r] 
					& wb^{c}_{1}(E) \ar[r] \ar[dl]|{\theta _{1}^{E}} & \cdots  \\						
					&&  E &&}
\end{align}
We will call \eqref{F.birtow} the weakly birational tower of $E$.

\begin{rmk}
	By \ref{prop.adj-ort.b} and \ref{def.birat.cover}, the weakly birational tower \eqref{F.birtow} is functorial with respect to
	morphisms in $\stablehomotopy$.
\end{rmk}

\begin{thm}  \label{thm.birspecseq}
	Let $G, K$ be arbitrary symmetric $T$-spectra in $\stablehomotopy$.  Then
	there is a spectral sequence of homological type with term $E^{1}_{p,q}=\Hom _{\stablehomotopy}(G, \susp{q-p}wb_{p/p-1}K)$
	and where the abutment is given by the associated graded group for the filtration $F_{\bullet}$ of $\Hom _{\stablehomotopy} (G,K)$ defined
	by the image of $\theta _{p\ast}^{K}:\Hom _{\stablehomotopy}(G, wb^{c}_{p}K)\rightarrow \Hom _{\stablehomotopy}(G, K)$, or
	equivalently the kernel of $\pi_{>q}:Hom _{\stablehomotopy}(G, K)\rightarrow Hom _{\stablehomotopy}(G, wb_{>q}K)$.
\end{thm}
\begin{proof}
	Since $\stablehomotopy$ is a triangulated category, the result follows from \ref{thm.3.1.octahedral-axiom} and
	\ref{F.birtow}.
\end{proof}

\begin{rmk}  \label{rmk.bi-wbirtss}
	The constructions of this section can be carried out as well using the tower (see \ref{def.locbirat-gens}):
	\begin{align} 
			\cdots \subseteq B_{q-1}^{\perp}
			\subseteq B_{q}^{\perp}
			\subseteq B_{q+1}^{\perp} \subseteq \cdots
		\end{align}
	and it follows from \ref{thm.ort=wbirationalTspectra} that both constructions are canonically isomorphic
	when the base scheme $X$ is of the form $\spec{k}$ with $k$ a perfect field.
\end{rmk}
\section{Rational Coefficients}  \label{Q.coeffs}

The main result of this section is \ref{thm.birtowQloc}, which describes
 the weakly birational tower with rational coefficients. 

\subsection{$\mathbb Q$-local spectra} \label{Q.local}

Recall that a symmetric 
$T$-spectrum $E$ in $\Tspectra$ is \emph{$\mathbb Q$-local} if the abelian group
$\Hom _{\stablehomotopy}(\susp{p}(\Gm ^{q}\wedge Y_{+}),E)$ is a $\mathbb Q$-vector space
for every $Y$ in $Sm_{X}$ and every $p,$ $q$ in $\mathbb Z$  \cite[Remark 4.3.3]{MR2061856},
and that a map $f:E\rightarrow F$ in $\Tspectra$ is a \emph{rational weak equivalence} if
for every $Y$ in $Sm_{X}$ and every $p,$ $q$ in $\mathbb Z$ the induced map
$f_{\ast}\otimes \mathbb Q: \Hom _{\stablehomotopy}(\susp{p}(\Gm ^{q}\wedge Y_{+}),E)\otimes \mathbb Q
\rightarrow \Hom _{\stablehomotopy}(\susp{p}(\Gm ^{q}\wedge Y_{+}),F)\otimes \mathbb Q$
is an isomorphism of $\mathbb Q$-vector spaces \cite[Remark 4.3.3]{MR2061856}.

The work of Morel \cite[\S 5.2 p.246 and Remark 4.3.3]{MR2061856} shows that the left Bousfield localization of 
$\Tspectra$ with respect to the rational weak equivalences exists.
We will write $\Tspectra _{\mathbb Q}$ for this left Bousfield localization, and
 $\stablehomotopy _{\mathbb Q}$ for its homotopy category.  

Recall that $\sphere$ is the sphere spectrum in $\stablehomotopy$.
The rational Moore spectrum $\ratMoore$  
in $\Tspectra$ is defined to be the homotopy colimit of the filtering diagram:
\[	\xymatrix{\sphere \ar[r]^-{2}&\sphere \ar[r]^-{3}&\cdots \ar[r]^-{n}&\sphere \ar[r]^-{n+1}&\cdots}
\]
where $\sphere \stackrel{n}{\rightarrow} \sphere$ is the composition of the sum map with the diagonal
$\sphere \stackrel{\Delta}{\rightarrow} \oplus _{i=1}^{n}\sphere \stackrel{\sum}{\rightarrow} \sphere$.
Let $u:\sphere \rightarrow \ratMoore$ be the canonical map.
A concrete model for a $\mathbb Q$-localization functor is given by $E\mapsto E\wedge \ratMoore$, i.e.
$E\wedge \ratMoore$ is $\mathbb Q$-local in $\stablehomotopy$ and the map 
$E\wedge (\sphere \stackrel{u}{\rightarrow}\ratMoore)$
is a rational weak equivalence.

\begin{lem}  \label{lem.detectQloc}
Let $E$ be a $\mathbb Q$-local symmetric $T$-spectrum in $\stablehomotopy$.  Then
the map $E\wedge u:E\rightarrow E\wedge \ratMoore$ is an isomorphism in $\stablehomotopy$.
\end{lem}
\begin{proof}
 Since $\stablehomotopy$ is a 
compactly generated category in the sense of Neeman (see \ref{subsubsec.genSH});
it suffices to show that for every compact generator $G\in \generators _{\stablehomotopy}$ (see \ref{gens.SH})
the map induced by $E\wedge u$, 
$\Hom_{\stablehomotopy}(G,E)\stackrel{}{\rightarrow} \Hom_{\stablehomotopy}(G,E\wedge \ratMoore)$
is an isomorphism of abelian groups.
We observe that
$E\wedge  \ratMoore$ is the homotopy colimit of the diagram:
\[	\xymatrix{E \ar[r]^-{2}&E \ar[r]^-{3}&\cdots \ar[r]^-{n}&E \ar[r]^-{n+1}&\cdots}
\]
Thus, by \cite[Lem. 2.8]{MR1308405} we just need to show that for every integer $n\geq 1$, 
the maps $\Hom_{\stablehomotopy}(G,E)\stackrel{n}{\rightarrow}\Hom_{\stablehomotopy}(G,E)$ are isomorphisms
of abelian groups.  This follows directly from the fact that $E$ is $\mathbb Q$-local (see \ref{Q.local}).
\end{proof}

By construction, $\ratMoore$ is $\mathbb Q$-local in $\stablehomotopy$.  Thus, by \ref{lem.detectQloc} we conclude 
that $\ratMoore \wedge \ratMoore \cong \ratMoore$ in $\stablehomotopy$.  Hence, we deduce that $\ratMoore$
is a commutative ring spectrum in $\stablehomotopy$; with unit $u:\sphere \rightarrow \ratMoore$.  

\begin{rmk}  \label{rmk.1Qstrictring}
Combining \cite[Prop. 1.4.3]{MR2900540} and \cite[p.99 (2.3.22.1)-(2.3.22.2)]{MR2900540}; we can assume
that $\ratMoore$ is a commutative ring spectrum with unit in $\Tspectra$.
\end{rmk}

\begin{prop}  \label{rmk.Qlocequivcond}
The following conditions on a symmetric $T$-spectrum $E$ are equivalent:
\begin{enumerate}  
	\item \label{rmk.Qlocequivcond.1} $E$ is $\mathbb Q$-local (see \ref{Q.local}),
	\item \label{rmk.Qlocequivcond.2} The map $E\wedge u:E\rightarrow E\wedge \ratMoore$ is an 
			isomorphism in $\stablehomotopy$.
	\item \label{rmk.Qlocequivcond.3} $E$ is a module in $\stablehomotopy$ over $\ratMoore$.
	\item \label{rmk.Qlocequivcond.4} There exists a $\ratMoore$-module in $\Tspectra$, $E'$ which is
				isomorphic to $E$ in $\stablehomotopy$.
\end{enumerate}
\end{prop}
\begin{proof}
\eqref{rmk.Qlocequivcond.1}$\Rightarrow$\eqref{rmk.Qlocequivcond.2}: This follows from \ref{lem.detectQloc}.

\eqref{rmk.Qlocequivcond.2}$\Rightarrow$\eqref{rmk.Qlocequivcond.4}: It suffices to show
that $E'=E\wedge \ratMoore $ is a $\ratMoore$-module in $\Tspectra$, and this holds
since $\ratMoore$ is a commutive ring spectrum with unit in $\Tspectra$ (see \ref{rmk.1Qstrictring}).

\eqref{rmk.Qlocequivcond.4}$\Rightarrow$\eqref{rmk.Qlocequivcond.3}: Obvious.

\eqref{rmk.Qlocequivcond.3}$\Rightarrow$\eqref{rmk.Qlocequivcond.1}: Let $\mu_{E}:E\wedge \ratMoore
\rightarrow E$ be the map inducing the $\ratMoore$-module structure of $E$ in $\stablehomotopy$.
We observe that the composition $\mu_{E}\circ (E\wedge u): E\rightarrow E$ is the identity on $E$.
Hence $E$ is a direct summand of $E\wedge \ratMoore$, which is $\mathbb Q$-local by construction.
Thus we conclude that $E$ is also $\mathbb Q$-local, since this property is clearly closed under
direct summands (see \ref{Q.local}).
\end{proof}

\subsubsection{Orthogonality and weakly birational covers}

\begin{lem}  \label{rat.comp.orth}
	Let $E$ be an arbitrary symmetric $T$-spectrum in $\stablehomotopy$.  Assume that
	$E$ belongs to $\northogonal{q}$  for some integer
	$q\in \mathbb Z$ (see \ref{def.orth}).  Then $E\wedge \ratMoore$ also belongs to $\northogonal{q}$.
\end{lem}
\begin{proof}
	Combining \ref{def.orth}, \ref{cor.r-orth.loc} and \ref{Vov.eff} we deduce that
	$\northogonal{q}$ is a localizing subcategory of $\stablehomotopy$.
	By construction, $E\wedge  \ratMoore$ is the homotopy colimit of the diagram:
\[	\xymatrix{E \ar[r]^-{2}&E \ar[r]^-{3}&\cdots \ar[r]^-{n}&E \ar[r]^-{n+1}&\cdots}
\]
	To conclude, we observe that by definition of the homotopy colimit (see \cite[Def. 1.6.4]{MR1812507}) any 
	localizing subcategory (see \ref{subsec.locsubcats}) is closed under homotopy colimits.
\end{proof}

Given an arbitrary integer $n\in \mathbb Z$, consider the weakly $n$-birational cover of $E$ 
(see \ref{def.birat.cover}, \ref{prop.3.1.counit-properties}),
$\theta^{E}_{n}:wb^{c}_{n}(E)\rightarrow E$.  Assume that
$E$ is a $\mathbb Q$-local in $\stablehomotopy$.
By \ref{rmk.Qlocequivcond}, we
conclude that $E \wedge \ratMoore \cong E$ in $\stablehomotopy$.
Thus, taking the $\mathbb Q$-localization of $\theta^{E}_{n}$ we obtain a commutative triangle
in $\stablehomotopy$:
\begin{align}	\label{diag.com.Qloc1}
	\xymatrix{wb^{c}_{n}(E)\ar[r]^-{\theta^{E}_{n}} \ar[d]_-{wb^{c}_{n}(E) \wedge u}& E\\
		wb^{c}_{n}(E)\wedge \ratMoore \ar[ur]_-{\theta^{E}_{n}\wedge \ratMoore}&}
\end{align}
Combining \ref{rat.comp.orth}, \ref{thm.ort=wbirationalTspectra} and
the universal property of $\theta^{E}_{n}$ (see
\ref{prop.3.1.counit-properties}),
we conclude that there is a map $r: wb^{c}_{n}(E)\wedge \ratMoore \rightarrow wb^{c}_{n}(E)$
such that the following triangle in $\stablehomotopy$ commutes:
\begin{align}	\label{diag.com.Qloc2}
	\xymatrix{wb^{c}_{n}(E)\ar[r]^-{\theta^{E}_{n}} & E\\
		wb^{c}_{n}(E)\wedge \ratMoore \ar[ur]_-{\theta^{E}_{n}\wedge \ratMoore} \ar[u]^-{r}&}
\end{align}

\begin{lem}  \label{lem.comp.id}
Let $E$ be $\mathbb Q$-local spectrum in $\stablehomotopy$.  Then
the composition map $r\circ (wb^{c}_{n}(E)\wedge u):wb^{c}_{n}(E)\rightarrow wb^{c}_{n}(E)$
is  the identity on $wb^{c}_{n}(E)$, i.e. $wb^{c}_{n}(E)$ is a direct summand of $wb^{c}_{n}(E)\wedge \ratMoore$.
Hence, $wb^{c}_{n}(E)$ is $\mathbb Q$-local and the map 
$wb^{c}_{n}(E) \wedge u: wb^{c}_{n}(E)\rightarrow wb^{c}_{n}(E)\wedge \ratMoore$ is an isomorphism in $\stablehomotopy$.
\end{lem}
\begin{proof}
Since the direct summand of a $\mathbb Q$-local object in $\stablehomotopy$ is $\mathbb Q$-local (see \ref{Q.local}),
by \ref{rmk.Qlocequivcond} it suffices to show that $r\circ (wb^{c}_{n}(E)\wedge u)=id$.  
This follows from the universal property 
\eqref{prop.3.1.counit-properties} of $\theta^{E}_{n}$ and the commutativity of 
\eqref{diag.com.Qloc2}, \eqref{diag.com.Qloc1}.
\end{proof}

\subsubsection{Voevodsky's slice tower for orthogonal spectra}  \label{subsubsec.slice.Qorth}
 
Recall that the $q$-effective motivic stable homotopy category $\neffstablehomotopy{q}$ 
(see \ref{Vov.eff}) is the localizing subcategory of $\stablehomotopy$, of the form $Loc(C^{q}_{\mathit{eff}})$
(see \ref{eq.localizing.gen} and \ref{eqn.Cqeff}).
Voevodsky \cite{MR1977582} defines the slice filtration in $\stablehomotopy$ to be the  family of localizing subcategories:
		\[ \cdots \subseteq \neffstablehomotopy{q+1} \subseteq \neffstablehomotopy{q}
			\subseteq \neffstablehomotopy{q-1} \subseteq \cdots
		\]
The work of Neeman \cite{MR1308405}, \cite{MR1812507} shows that the inclusion
$\neffstablehomotopy{q}\stackrel{i_{q}}{\rightarrow} \stablehomotopy$ admits
a right adjoint $\stablehomotopy \stackrel{r_{q}}{\rightarrow} \neffstablehomotopy{q}$ \cite[Prop. 3.1.12]{MR2807904}; 
and that the functors, $\stablehomotopy \stackrel{f_{q}}{\rightarrow} \stablehomotopy ,$ 
$\stablehomotopy \stackrel{s_{q}}{\rightarrow} \stablehomotopy$ are triangulated, where $f_{q}$ is 
defined as the composition $i_{q}\circ r_{q}$; and $s_{q}$ 
is characterized by the fact that for every $E\in \stablehomotopy$, there exists a
distinguished triangle in $\stablehomotopy$ \cite[Thms. 3.1.16, 3.1.18]{MR2807904}:
		\begin{align}
				\label{orth.dist.triang2}	
			f_{q+1}E \rightarrow f_{q}E \rightarrow s_{q}E.
		\end{align}
We will refer to $f_{q}E$ as the \emph{$(q-1)$-connective cover} of $E$,  
and to $s_{q}E$ as the \emph{$q$-slice} of $E$.

\begin{rmk}  \label{rmk.sliceort}
It follows directly from the definition that $s_{q}E$ is in $\northogonal{q+1}$.
\end{rmk}

\begin{rmk}  \label{rmk.slice.hocolim}
Since $\stablehomotopy$ is compactly generated in the sense of Neeman (see \ref{subsubsec.genSH})
with set of compact generators $\generators _{\stablehomotopy}$ (see \ref{gens.SH}),
and by construction $\generators _{\stablehomotopy}$ is the union of $C^{q}_{\mathit{eff}}$
for $q\in \mathbb Z$ (see \ref{eqn.Cqeff}), we conclude that $E$ is the homotopy colimit of its
slice tower:
$	\cdots \rightarrow f_{n}E \rightarrow f_{n-1}E\rightarrow f_{n-2}E \rightarrow  \cdots .$
\end{rmk}
	
The sphere spectrum $\sphere$ is contained in the localizing subcategory $\stablehomotopyeff$ (see \ref{Vov.eff}). 
Thus, $\ratMoore$ is also contained in $\stablehomotopyeff$ since localizing subcategories
are closed under homotopy colimits.
Consider the triangulated functor, $-\wedge \ratMoore: \stablehomotopy \rightarrow \stablehomotopy$.

\begin{thm}		\label{lem.Qloccompslice}
Let $E$ be an arbitrary symmetric $T$-spectrum in $\stablehomotopy$, and $p\in \mathbb Z$
an arbitrary integer.  Then $s_{p}(E\wedge \ratMoore)\cong s_{p}(E)\wedge \ratMoore$.
\end{thm}
\begin{proof}
Since $\ratMoore$ belongs to $\stablehomotopyeff$, we deduce that $(\neffstablehomotopy{q})\wedge \ratMoore
\subseteq (\neffstablehomotopy{q})$ for every integer $q\in \mathbb Z$.  Furthermore, it is clear that the
functor $-\wedge \ratMoore$ commutes with filtered homotopy colimits.

Therefore, by \cite[Rmk. 2.13 and Thm. 2.12]{MR3034283} it suffices to show that $s_{q}(E)\wedge \ratMoore$ is in
$\northogonal{q+1}$ for every integer $q\in \mathbb Z$.  This follows from \ref{rat.comp.orth} since
by construction $s_{q}(E)$ belongs to $\northogonal{q+1}$ (see \ref{rmk.sliceort}).
\end{proof}

\begin{prop}  \label{prop.slice.HQmod}
Let $E$ be a $\mathbb Q$-local symmetric $T$-spectrum in $\stablehomotopy$.  Then
for every integer $p\in \mathbb Z$, $s_{p}E$ is a module in $\stablehomotopy$ over 
$s_{0}(\sphere)\wedge \ratMoore$.
\end{prop}
\begin{proof}
We observe that $E\cong E\wedge \ratMoore$ in $\stablehomotopy$, since $E$ is $\mathbb Q$-local
in $\stablehomotopy$ (see \ref{rmk.Qlocequivcond}).  Thus, by \ref{lem.Qloccompslice} 
we deduce that $s_{p}E\cong s_{p}(E\wedge \ratMoore)\cong s_{p}(E)\wedge \ratMoore$ in $\stablehomotopy$.

To finish the proof, it suffices to show that $s_{p}E$ is a module
in $\stablehomotopy$ over $s_{0}\sphere$.  This follows from
\cite[Thm. 3.6.14(6)]{MR2807904}.
\end{proof}

Let $\beilinson=KGL ^{(0)} \in \Tspectra$ 
denote the Beilinson motivic
cohomology spectrum constructed by Riou in \cite{MR2651359}.  By the work of
Cisinski-D\'eglise \cite[Cor. 14.2.6]{Cisinski:2009fk}, we conclude that $\beilinson$ is a commutative
cofibrant ring spectrum in $\Tspectra$.  In addition, in \cite[Thm. 4.1]{MR2861232} we show that
$s_{0}(\sphere)\wedge \ratMoore$ is equipped with a unique structure of
$\beilinson$-algebra in $Spt (\mathcal M _{X})$.

\begin{cor}  \label{cor.slice.HBmod}
Let $E$ be a $\mathbb Q$-local symmetric $T$-spectrum in $\stablehomotopy$.  Then
for every integer $p\in \mathbb Z$, $s_{p}E$ is a module in $\Tspectra$ over 
$\beilinson$.
\end{cor}
\begin{proof}
Combining \ref{prop.slice.HQmod} and \cite[Thm. 4.1]{MR2861232}, we deduce
that $s_{p}E$ is a module in $\stablehomotopy$ over $\beilinson$.  Then the result follows
from \cite[Cor 14.2.16(i)-(v)]{Cisinski:2009fk}, since $E$ is $\mathbb Q$-local in $\stablehomotopy$ and hence
$s_{p}E\cong s_{p}(E\wedge \ratMoore)
\cong s_{p}(E)\wedge \ratMoore$ (see \ref{lem.Qloccompslice}) is also $\mathbb Q$-local.
\end{proof}

Let $e:\sphere \rightarrow \beilinson$ be the unit map of the ring spectrum $\beilinson$ in $\Tspectra$.

\begin{thm}  \label{rat.orth.motive}
Let $E$  be a $\mathbb Q$-local symmetric $T$-spectrum in $\stablehomotopy$.  Assume that
$E$ belongs to the localizing subcategory $\northogonal{n}$ for some integer $n\in \mathbb Z$. 
Then $E$ is a module in $\Tspectra$ over $\beilinson$.
\end{thm}
\begin{proof}
By  \cite[Cor 14.2.16(i)-(v)]{Cisinski:2009fk}, it suffices to show that $E\wedge e:E\rightarrow E\wedge \beilinson$
is an isomorphism in $\stablehomotopy$.  Since $E\cong \hocolim _{p\leq n}f_{p}E$
in $\stablehomotopy$ (see \ref{rmk.slice.hocolim}), it suffices to show that
$f_{p}E\wedge e$ is an isomorphism in $\stablehomotopy$ for every integer $p\leq n$.

We observe that $f_{p}E \cong \ast \cong s_{p}E$ for $p\geq n$, since by hypothesis $E$ is in $\northogonal{n}$.
Hence, the slice tower of $E$ is of the form:
\[	\ast=f_{n}E \rightarrow s_{n-1}E\cong f_{n-1}E\rightarrow f_{n-2}E \rightarrow f_{n-3}E\rightarrow \cdots
\]
Thus, by \ref{orth.dist.triang2} it suffices to show that  $s_{p}E\wedge e$ is an isomorphism in $\stablehomotopy$
for every integer $p\leq n$.  This follows from \ref{cor.slice.HBmod} and \cite[Cor 14.2.16(i)-(v)]{Cisinski:2009fk}.
\end{proof}

\subsection{The weakly birational tower for $\mathbb Q$-local spectra}

Given symmetric $T$-spectrum $E$ in $\stablehomotopy$, we
consider its weakly birational tower \eqref{F.birtow}:
\begin{align*}	
	\xymatrix@C=1.5pc{\cdots \ar[r]
					& \ar[dr]|{\theta _{-1}^{E}} wb^{c}_{-1}(E) 
					\ar[r] & \ar[d]|{\theta _{0}^{E}} wb^{c}_{0}(E) \ar[r] 
					& wb^{c}_{1}(E) \ar[r] \ar[dl]|{\theta _{1}^{E}} & \cdots  \\						
					&&  E &&}
\end{align*}

Recall that $wb^{c}_{n}E$ is the weakly $n$-birational cover of $E$
(see \ref{def.birat.cover}, \ref{prop.3.1.counit-properties}).

\begin{thm}  \label{thm.birtowQloc}
Let $E$ be a $\mathbb Q$-local symmetric $T$-spectrum in $\stablehomotopy$ (see \ref{Q.local}),
and $n\in \mathbb Z$ an arbitrary integer.  Then $wb^{c}_{n}E$ satisfies the following properties:
	\begin{enumerate}
		\item  \label{thm.birtowQloc.1} $wb^{c}_{n}E$ is $\mathbb Q$-local in $\stablehomotopy$.
		\item	\label{thm.birtowQloc.2} $wb^{c}_{n}E$ is a module in $\Tspectra$ over $\beilinson$.
		\item \label{thm.birtowQloc.3} The natural map $wb^{c}_{n}E\rightarrow wb^{c}_{n+1}E$ is a map of
				$\beilinson$-modules in $\Tspectra$.
	\end{enumerate}
\end{thm}
\begin{proof}
\eqref{thm.birtowQloc.1}: This follows from \ref{lem.comp.id}.

\eqref{thm.birtowQloc.2}: By construction $wb^{c}_{n}E$ is in $WB_{n}^{\perp}$
(see \ref{prop.adj-ort.b} and \ref{def.birat.cover}).  Hence, the result follows by
combining \ref{rat.orth.motive} and \ref{thm.ort=wbirationalTspectra}.

\eqref{thm.birtowQloc.3}: By \cite[Cor. 14.2.6]{Cisinski:2009fk} $\beilinson$ is a commutative ring in $\Tspectra$,
with unit $u:\sphere \rightarrow \beilinson$.  Thus
the bottom row in the following commutative diagram in $\stablehomotopy$ is a map of
$\beilinson$-modules in $\Tspectra$:
\[ \xymatrix{wb^{c}_{n}E \ar[r] \ar[d]_-{(wb^{c}_{n}E)\wedge e}& wb^{c}_{n+1}E \ar[d]^-{(wb^{c}_{n+1}E)\wedge e}\\
				wb^{c}_{n}E\wedge \beilinson \ar[r]& wb^{c}_{n+1}E\wedge \beilinson}
\]
Hence it suffices to show that the vertical maps are isomorphisms in $\stablehomotopy$.  This follows from 
\cite[Cor 14.2.16(i)-(v)]{Cisinski:2009fk}, since by \eqref{thm.birtowQloc.2} above $wb^{c}_{n}E$ and
$wb^{c}_{n+1}E$ are modules in $\Tspectra$ over $\beilinson$.
\end{proof}

\section{The Weakly Birational Tower for Motivic Cohomology}  \label{sec.biratHZ}

Throughout this section, the base scheme $X$ will be of the form $\spec{k}$ with $k$ a perfect field.
We will study the weakly birational tower \eqref{F.birtow} for Voevodsky's $T$-spectrum $HR$,
which represents in $\stablehomotopy$ motivic cohomology \cite[6.1]{MR1648048} with coefficients in
an abelian group or a commutative ring $R$.  Our goal is to show that in some cases the tower is
not trivial, and that it induces and interesting finite filtration on the Chow groups.
						
\subsection{Main properties}

\begin{lem}  \label{orth.HZ}
	$\hr$ belongs to $WB_{0}^{\perp}$ and to $\stablehomotopyeff$.
\end{lem}
\begin{proof}
By \ref{thm.newpres.birat.b} and
\cite[Def. 3.1 and Thm. 3.7]{MR3035769}, we are reduced to show that $\hr$ is isomorphic
in $\stablehomotopy$ to its zero slice $s_{0}(\hr)$ \cite{MR1977582}.  This follows from
\cite{MR2101286} in characteristic zero, and \cite[Lem. 10.4.1]{MR2365658} in positive characteristic.
\end{proof}

\begin{prop}  \label{prop.hzgeq0birat}
	Let $n\geq 0$ be an arbitrary integer.  Then the natural map (see \ref{F.birtow}),
		$	\theta _{n}^{\hr}:wb ^{c}_{n}(\hr)\rightarrow \hr$
	is an isomorphism in $\stablehomotopy$.
\end{prop}
\begin{proof}
Combining \ref{cor.comp.tool1} and \ref{eq.birtower}, it suffices to show that $\hr$ is in 
$WB_{0}^{\perp}$.  This follows from \ref{orth.HZ}.
\end{proof}

Hence, we conclude that the weakly birational tower \eqref{F.birtow} for motivic cohomology with $R$ coefficients is
as follows:

\begin{align}	\label{motbirtower}
	\xymatrix@C=1.5pc{\cdots \ar[r]
										& wb^{c} _{-3}(\hr) \ar[drr]|{\theta _{-3}^{\hr}} \ar[r]
										& wb^{c} _{-2}(\hr) \ar[r] \ar[dr]|{\theta _{-2}^{\hr}}& wb^{c}_{-1}(\hr) 
										\ar[d]|{\theta _{-1}^{\hr}}  \\						
										&&& \hr}
\end{align}

\subsubsection{Non-triviality of the tower}
Our goal is to show that for $n<0$, the weakly birational covers $wb^{c} _{n}(\hr)$ and the layers
$wb_{n/(n-1)}(\hr)$ are interesting, i.e. they are not isomorphic in $\stablehomotopy$ to either $\hr$ or
$\ast$.  

First we introduce some notation; let $\hr ^{p,q}$, $A_{\hr}$ denote 
$\susp{p-q}\wedge \Gm ^{q} \wedge \hr$ and $\oplus _{p,q\in \mathbb Z}\hr ^{p,q}$, respectively.
By adjointness, there are isomorphisms:
\begin{align}  
\Hom _{\stablehomotopy}(\hr ^{p,q}, \hr) & \cong \Hom _{\stablehomotopy}(\hr , \hr ^{-p,-q}) =\mathcal A ^{-p,-q}
\label{adj.motoper} \\
\Hom _{\stablehomotopy}(A_{\hr}, \hr) & \cong \prod _{p,q\in \mathbb Z} 
\Hom _{\stablehomotopy}(\hr , \hr ^{-p,-q}) =\prod _{p,q\in \mathbb Z}\mathcal A ^{-p,-q}  \label{adj.motopertot}
\end{align}
where $\mathcal A ^{-p,-q}$ is the group of bistable operations in motivic cohomology with $R$
coefficients of degree $-p$ and weight $-q$ \cite{MR2031198}.

The non-triviality of $wb^{c} _{n}(\hr)$ and 
$wb_{n/(n-1)}(\hr)$ will follow from the existence of non-trivial elements in $\mathcal A ^{p,q}$
for appropriate $p$ and $q$.
This information is codified in the spectral sequence of the tower \eqref{motbirtower} evaluated in $A _{\hr}$.
Namely:

\begin{thm}  \label{thm.motoperspecseq}
Let $p\leq 0$.
Consider the spectral sequence of the tower \eqref{motbirtower} with term (see \ref{thm.birspecseq}):
\[ E^{1}_{p,q}=\Hom _{\stablehomotopy}(A _{\hr},\susp{q-p} wb_{p/p-1}\hr)=\prod _{j\in \mathbb Z} \mathcal A ^{j,-p}
\]
given by all bistable operations in motivic cohomology with $R$ coefficients of weight $-p$.
Then the spectral sequence degenerates.  Therefore, the term $E^{1}_{p,q}$ is the associated graded
$gr_{p}F_{\bullet}$ for the filtration $F_{\bullet}$ of $\Hom _{\stablehomotopy} (A _{\hr},\hr)=
\prod _{j,k\in \mathbb Z}\mathcal A ^{-j,-k}$ 
given by the image of 
\[ \theta _{n\ast}^{\hr}:\Hom _{\stablehomotopy}(A_{\hr}, wb^{c}_{n}\hr) \rightarrow 
\Hom _{\stablehomotopy}(A_{\hr}, \hr), \]
which can be written explicitly as $\prod _{j\in \mathbb Z, k\leq n}\mathcal A ^{-j,-k}$, i.e. coincides with all
bistable operations in motivic cohomology with $R$ coefficients of weight $\geq-n$.
\end{thm}
\begin{proof}
The existence of the spectral sequence with  term
\[ E^{1}_{p,q}=\Hom _{\stablehomotopy}(A _{\hr},\susp{q-p} wb_{p/p-1}\hr),\] 
follows from \ref{thm.birspecseq}.  Combining \ref{negoperbirat}, \ref{operbiratvanish} and \ref{lem.specseqdegs};
we deduce that the spectral sequence degenerates and that (see \ref{adj.motoper} and \ref{adj.motopertot}):
\[ E^{1}_{p,q}=\prod _{j\in \mathbb Z}\Hom _{\stablehomotopy}(\hr^{j,p},\hr)=
\prod _{j\in \mathbb Z} \mathcal A ^{j,-p}\]
Since the spectral sequence degenerates, we conclude that $E^{1}_{p,q}$ coincides with the
associated graded $gr_{p}F_{\bullet}$ for the filtration given by the image of $\theta _{n\ast}^{\hr}$.
Finally, combining \ref{negoperbirat}, \ref{operbiratvanish} and \eqref{adj.motoper}-\eqref{adj.motopertot}
we deduce that the image of $\theta _{n\ast}^{\hr}$ is $\prod _{j\in \mathbb Z, k\leq n}\mathcal A ^{-j,-k}$.
\end{proof}

\begin{thm}  \label{thm.hzcoversnotriv}
	Let $q>0$ be an arbitrary integer.  Assume that there is an integer $p$ such that 
	$\mathcal A^{p,q}\neq 0$, i.e. there exist non-trivial
	bistable motivic operations of degree $p$ and weight $q$.  Then:
	\begin{enumerate}
	\item	\label{thm.hzcoversnotriv.a}  For every integer $0<n\leq q$, $wb^{c}_{-n}\hr$ is not isomorphic to 
				$\ast$.  In addition, the natural map $\theta ^{\hr}_{-n}:wb^{c}_{-n}\hr \rightarrow \hr$ is not
				an isomorphism in $\stablehomotopy$.
	\item	\label{thm.hzcoversnotriv.b}  The layer $wb_{-q/(-q-1)}\hr$ is not isomorphic to $\ast$
				in $\stablehomotopy$.
	\end{enumerate}
\end{thm}
\begin{proof}
\eqref{thm.hzcoversnotriv.a}: On the one hand, combining \ref{negoperbirat} with \ref{adj.motoper}, we conclude that
$wb^{c}_{-n}\hr$ and $\ast$ are not isomorphic in $\stablehomotopy$.  On the other hand, by \ref{cor.negcoversnotiso}
we deduce that $\theta ^{\hr}_{-n}$ is not an isomorphism in $\stablehomotopy$.

\eqref{thm.hzcoversnotriv.b}: This follows by combining \ref{negoperbirat} and \ref{lem.specseqdegs}.
\end{proof}

\subsection{The filtration in the motivic cohomology of a scheme}

Recall that $X$ is of the form $\spec{k}$ with $k$ a perfect field and 
that $H^{p,q}=\susp{p-q}\wedge \Gm ^{q} \wedge \hr$.  

Let $Y\in Sm_{X}$.  We will write $H^{p,q}(Y,R)$ for $\Hom _{\stablehomotopy}(\Sigma _{T}^{\infty}Y_{+}, \hr ^{p,q})$,
i.e. for the motivic cohomology of $Y$ with coefficients in 
$R$ of degree $p$ and weight $q$.  By adjointness:
$H^{p,q}(Y,R)\cong \Hom _{\stablehomotopy}(\susp{-p+q}(\Gm ^{-q}\wedge Y_{+}), \hr )$.
Therefore, the abutment of the spectral sequence of the tower \eqref{motbirtower} evaluated in 
$\susp{-p+q}(\Gm ^{-q}\wedge Y_{+})$ induces an interesting filtration in $H^{p,q}(Y,R)$.  Namely:

\begin{thm}  \label{thm.filt.motcohY}
Let $p$, $q$ be arbitrary integers, and let $Y$ be in $Sm_{X}$.  Then there exists a decreasing filtration
$F^{\bullet}$ on $H^{p,q}(Y,R)$ where the $n$-component is given by the image of $\theta _{-n\ast}^{\hr}$:
\[ \xymatrix{\Hom _{\stablehomotopy}(\susp{-p+q}(\Gm ^{-q}\wedge Y_{+}), wb^{c}_{-n}\hr) \ar[d]^-{\theta _{-n\ast}^{\hr}}\\
\Hom _{\stablehomotopy}(\susp{-p+q}(\Gm ^{-q}\wedge Y_{+}), \hr)=H^{p,q}(Y,R).} \]
In addition, the filtration $F^{\bullet}$ is functorial in $Y$ with respect to morphisms in $Sm_{X}$.
\end{thm}
\begin{proof}
The existence of the filtration follows directly from \ref{thm.birspecseq}.  The functoriality is clear from
the definition.
\end{proof}

\subsubsection{Finiteness of the filtration}

Our goal is to show that for $Y$ in $Sm_{X}$, the filtration $F^{\bullet}$ on $H^{p,q}(Y,R)$ defined in
\ref{thm.filt.motcohY} is concentrated
in the range $0\leq n\leq q$.

\begin{thm}  \label{thm.filt.motcoh.finite}
Let $p$, $q$ be arbitrary integers, and let $Y$ be in $Sm_{X}$.  Then the decreasing filtration
$F^{\bullet}$ on $H^{p,q}(Y,R)$ constructed in \ref{thm.filt.motcohY} satisfies the following properties:
	\begin{enumerate}
		\item \label{thm.filt.motcoh.finite.a}  $F^{0}H^{p,q}(Y,R)=H^{p,q}(Y,R)$,
		\item \label{thm.filt.motcoh.finite.b}  $F^{q+1}H^{p,q}(Y,R)=0$.
	\end{enumerate}
\end{thm}
\begin{proof}
By \ref{prop.hzgeq0birat} we deduce that $\theta ^{\hr}_{0}$ is an isomorphism in $\stablehomotopy$.
This proves the first claim.  For the second claim, we observe that $\susp{-p+q}(\Gm ^{-q}\wedge Y_{+})$
is in $\neffstablehomotopy{-q}$ (see \ref{eqn.Cqeff}).  Thus, it suffices to show that $wb^{c}_{-q-1}\hr$ belongs to
$\northogonal{-q}$ (see \ref{def.orth}), or equivalently to $WB_{-q-1}^{\perp}$
(see \ref{thm.ort=wbirationalTspectra}). This follows from
\ref{prop.adj-ort.b} and \ref{def.birat.cover}.
\end{proof}

\subsection{Auxiliary results}

\begin{lem}  \label{noneghzoper}
	Let $q>0$ be an arbitrary integer.  Then,  for every
	integer $p$: $\Hom _{\stablehomotopy}(\hr^{p,q},\hr)=0$.
\end{lem}
\begin{proof}
By \ref{orth.HZ}, we conclude that $\hr^{p,q}$ is in $\neffstablehomotopy{q}$.  Combining \ref{thm.ort=wbirationalTspectra}
and \ref{def.orth}, we observe that it suffices
to show that $\hr$ is in $WB_{q-1}^{\perp}$.  But this follows from
\ref{cor.comp.tool1} and \ref{prop.hzgeq0birat}.
\end{proof}

\begin{lem}  \label{negoperbirat}
	Let $q\leq n$ be arbitrary integers.  Then for every integer p, the natural map; $\theta^{\hr}_{n}:wb^{c}_{n}\hr \rightarrow \hr$
	induces an isomorphism:
	\[	\Hom _{\stablehomotopy}(\hr^{p,q},wb^{c}_{n}\hr)\rightarrow \Hom _{\stablehomotopy}(\hr^{p,q}, \hr).
	\]
\end{lem}
\begin{proof}
By \ref{prop.3.1.counit-properties} and \ref{eq.birtower} it suffices to show that $\hr^{p,q}$ is in $WB_{q}^{\perp}$.
Then the result follows by combining \ref{orth.HZ} and \ref{lem.Gmshift.orth}.
\end{proof}

\begin{lem}  \label{operbiratvanish}
	Let $q\geq n+1$ be arbitrary integers.  Then, for every integer p:
	$\Hom _{\stablehomotopy}(\hr^{p,q},wb^{c}_{n}\hr)=0$.
\end{lem}
\begin{proof}
By construction $wb^{c}_{n}\hr$ is in $WB_{n}^{\perp}$ (see \ref{prop.adj-ort.b} and \ref{def.birat.cover}),
and therefore also in $WB_{q-1}^{\perp}$ (see \ref{eq.birtower}).  
Combining \ref{thm.ort=wbirationalTspectra}
and \ref{def.orth}, we conclude that it suffices
to show that $\hr^{p,q}$ is in $\neffstablehomotopy{q}$.  This follows from
\ref{orth.HZ}.
\end{proof}

\begin{cor}		\label{cor.negcoversnotiso}
Let $n<0$ be an arbitrary integer.  Then
the natural map $\theta ^{\hr}_{n}:wb^{c}_{n}\hr \rightarrow \hr$ is not an isomorphism in $\stablehomotopy$.
\end{cor}
\begin{proof}
Taking $p=q=0$ in
\ref{operbiratvanish} we deduce that the identity map on $\hr$ does not factor through $\theta ^{\hr}_{n}$.
Hence we conclude that $\theta ^{\hr}_{n}$ is not an isomorphism in $\stablehomotopy$.
\end{proof}

Consider the distinguished triangle in $\stablehomotopy$ (see \ref{thm.3.1.slicefiltration}):
	\[  \xymatrix{wb^{c}_{n}\hr \ar[r]& wb^{c}_{n+1}\hr \ar[r]^-{\pi_{n+1}}& 
							 wb_{n+1/n}\hr \ar[r]^-{\sigma_{n+1}}& \susp{1}wb^{c}_{n}\hr}
	\]
	
\begin{lem}  \label{lem.specseqdegs}
	Let $q$ be an arbitrary integer.  Then,  the natural map $\pi _{q}$
	induces an isomorphism:
	\[	\Hom _{\stablehomotopy}(\hr^{p,q},wb^{c}_{q}\hr)\rightarrow \Hom _{\stablehomotopy}(\hr^{p,q}, wb_{q/(q-1)}\hr).
	\]
\end{lem}
\begin{proof}
Since $\stablehomotopy$ is a triangulated category, it suffices to show that 
\[ \Hom _{\stablehomotopy}(\hr^{p,q}, wb^{c}_{q-1}\hr)=\Hom _{\stablehomotopy}(\hr^{p,q}, \susp{1}wb^{c}_{q-1}\hr)=0.
\]
On the one hand, by \ref{operbiratvanish} we conclude that $\Hom _{\stablehomotopy}(\hr^{p,q}, wb^{c}_{q-1}\hr)=0$.
On the other hand, by adjointness and \ref{operbiratvanish} we conclude that:
\[ \Hom _{\stablehomotopy}(\hr^{p,q}, \susp{1}wb^{c}_{q-1}\hr)\cong \Hom _{\stablehomotopy}(\hr^{p-1,q}, wb^{c}_{q-1}\hr)=0.
\]
This finishes the proof.
\end{proof}



\section*{Acknowledgements}
	The author would like to thank warmly Chuck Weibel for his interest in this work, as well as for
	all the advice and support during all these years.


\bibliography{biblio_locSH}

\begin{thebibliography}{10}

\bibitem{MR923131}
A.~A. Be{\u\i}linson.
\newblock Height pairing between algebraic cycles.
\newblock In {\em {$K$}-theory, arithmetic and geometry ({M}oscow,
  1984--1986)}, volume 1289 of {\em Lecture Notes in Math.}, pages 1--25.
  Springer, Berlin, 1987.

\bibitem{MR558224}
S.~Bloch.
\newblock {\em Lectures on algebraic cycles}.
\newblock Duke University Mathematics Series, IV. Duke University Mathematics
  Department, Durham, N.C., 1980.

\bibitem{Cisinski:2009fk}
D.-C. Cisinski and F.~D{\'e}glise.
\newblock Triangulated categories of mixed motives.
\newblock 2009.

\bibitem{MR2900540}
D.-C. Cisinski and F.~D{\'e}glise.
\newblock Mixed {W}eil cohomologies.
\newblock {\em Adv. Math.}, 230(1):55--130, 2012.

\bibitem{MR1944041}
P.~S. Hirschhorn.
\newblock {\em Model categories and their localizations}, volume~99 of {\em
  Mathematical Surveys and Monographs}.
\newblock American Mathematical Society, Providence, RI, 2003.

\bibitem{MR2197578}
J.~Hornbostel.
\newblock Localizations in motivic homotopy theory.
\newblock {\em Math. Proc. Cambridge Philos. Soc.}, 140(1):95--114, 2006.

\bibitem{MR1650134}
M.~Hovey.
\newblock {\em Model categories}, volume~63 of {\em Mathematical Surveys and
  Monographs}.
\newblock American Mathematical Society, Providence, RI, 1999.

\bibitem{MR1860878}
M.~Hovey.
\newblock Spectra and symmetric spectra in general model categories.
\newblock {\em J. Pure Appl. Algebra}, 165(1):63--127, 2001.

\bibitem{MR1265533}
U.~Jannsen.
\newblock Motivic sheaves and filtrations on {C}how groups.
\newblock In {\em Motives ({S}eattle, {WA}, 1991)}, volume~55 of {\em Proc.
  Sympos. Pure Math.}, pages 245--302. Amer. Math. Soc., Providence, RI, 1994.

\bibitem{MR1787949}
J.~F. Jardine.
\newblock Motivic symmetric spectra.
\newblock {\em Doc. Math.}, 5:445--553 (electronic), 2000.

\bibitem{MR2365658}
M.~Levine.
\newblock The homotopy coniveau tower.
\newblock {\em J. Topol.}, 1(1):217--267, 2008.

\bibitem{MR2061856}
F.~Morel.
\newblock On the motivic {$\pi_0$} of the sphere spectrum.
\newblock In {\em Axiomatic, enriched and motivic homotopy theory}, volume 131
  of {\em NATO Sci. Ser. II Math. Phys. Chem.}, pages 219--260. Kluwer Acad.
  Publ., Dordrecht, 2004.

\bibitem{MR1813224}
F.~Morel and V.~Voevodsky.
\newblock {${\bf A}^1$}-homotopy theory of schemes.
\newblock {\em Inst. Hautes {\'E}tudes Sci. Publ. Math.}, (90):45--143 (2001),
  1999.

\bibitem{MR1225267}
J.~P. Murre.
\newblock On a conjectural filtration on the {C}how groups of an algebraic
  variety. {I}. {T}he general conjectures and some examples.
\newblock {\em Indag. Math. (N.S.)}, 4(2):177--188, 1993.

\bibitem{MR1308405}
A.~Neeman.
\newblock The {G}rothendieck duality theorem via {B}ousfield's techniques and
  {B}rown representability.
\newblock {\em J. Amer. Math. Soc.}, 9(1):205--236, 1996.

\bibitem{MR1812507}
A.~Neeman.
\newblock {\em Triangulated categories}, volume 148 of {\em Annals of
  Mathematics Studies}.
\newblock Princeton University Press, Princeton, NJ, 2001.

\bibitem{MR2794632}
A.~Neeman.
\newblock Colocalizing subcategories of {$\bold D(R)$}.
\newblock {\em J. Reine Angew. Math.}, 653:221--243, 2011.

\bibitem{MR992981}
Y.~P. Nesterenko and A.~A. Suslin.
\newblock Homology of the general linear group over a local ring, and
  {M}ilnor's {$K$}-theory.
\newblock {\em Izv. Akad. Nauk SSSR Ser. Mat.}, 53(1):121--146, 1989.

\bibitem{MR2807904}
P.~Pelaez.
\newblock Multiplicative properties of the slice filtration.
\newblock {\em Ast\'erisque}, (335):xvi+289, 2011.

\bibitem{MR2861232}
P.~Pelaez.
\newblock On the orientability of the slice filtration.
\newblock {\em Homology Homotopy Appl.}, 13(2):293--300, 2011.

\bibitem{MR3035769}
P.~Pelaez.
\newblock Birational motivic homotopy theories and the slice filtration.
\newblock {\em Doc. Math.}, 18:51--70, 2013.

\bibitem{MR3034283}
P.~Pelaez.
\newblock On the functoriality of the slice filtration.
\newblock {\em J. K-Theory}, 11(1):55--71, 2013.

\bibitem{Pelaezunstsl}
P.~Pelaez.
\newblock The unstable slice filtration.
\newblock {\em Trans. Amer. Math. Soc.}, 366(11):5991--6025, 2014.

\bibitem{MR0338129}
D.~Quillen.
\newblock Higher algebraic {$K$}-theory. {I}.
\newblock In {\em Algebraic {$K$}-theory, {I}: {H}igher {$K$}-theories ({P}roc.
  {C}onf., {B}attelle {M}emorial {I}nst., {S}eattle, {W}ash., 1972)}, pages
  85--147. Lecture Notes in Math., Vol. 341. Springer, Berlin, 1973.

\bibitem{MR0223432}
D.~G. Quillen.
\newblock {\em Homotopical algebra}.
\newblock Lecture Notes in Mathematics, No. 43. Springer-Verlag, Berlin, 1967.

\bibitem{MR2651359}
J.~Riou.
\newblock Algebraic {$K$}-theory, {${\bf A}^1$}-homotopy and {R}iemann-{R}och
  theorems.
\newblock {\em J. Topol.}, 3(2):229--264, 2010.

\bibitem{MR1744945}
A.~Suslin and V.~Voevodsky.
\newblock Bloch-{K}ato conjecture and motivic cohomology with finite
  coefficients.
\newblock In {\em The arithmetic and geometry of algebraic cycles ({B}anff,
  {AB}, 1998)}, volume 548 of {\em NATO Sci. Ser. C Math. Phys. Sci.}, pages
  117--189. Kluwer Acad. Publ., Dordrecht, 2000.

\bibitem{MR1764199}
A.~Suslin and V.~Voevodsky.
\newblock Relative cycles and {C}how sheaves.
\newblock In {\em Cycles, transfers, and motivic homology theories}, volume 143
  of {\em Ann. of Math. Stud.}, pages 10--86. Princeton Univ. Press, Princeton,
  NJ, 2000.

\bibitem{MR1648048}
V.~Voevodsky.
\newblock {$\bold A^1$}-homotopy theory.
\newblock In {\em Proceedings of the {I}nternational {C}ongress of
  {M}athematicians, {V}ol. {I} ({B}erlin, 1998)}, number Extra Vol. I, pages
  579--604 (electronic), 1998.

\bibitem{MR1764202}
V.~Voevodsky.
\newblock Triangulated categories of motives over a field.
\newblock In {\em Cycles, transfers, and motivic homology theories}, volume 143
  of {\em Ann. of Math. Stud.}, pages 188--238. Princeton Univ. Press,
  Princeton, NJ, 2000.

\bibitem{MR1883180}
V.~Voevodsky.
\newblock Motivic cohomology groups are isomorphic to higher {C}how groups in
  any characteristic.
\newblock {\em Int. Math. Res. Not.}, (7):351--355, 2002.

\bibitem{MR1977582}
V.~Voevodsky.
\newblock Open problems in the motivic stable homotopy theory. {I}.
\newblock In {\em Motives, polylogarithms and Hodge theory, Part I (Irvine, CA,
  1998)}, volume~3 of {\em Int. Press Lect. Ser.}, pages 3--34. Int. Press,
  Somerville, MA, 2002.

\bibitem{MR2031198}
V.~Voevodsky.
\newblock Reduced power operations in motivic cohomology.
\newblock {\em Publ. Math. Inst. Hautes \'Etudes Sci.}, (98):1--57, 2003.

\bibitem{MR2101286}
V.~Voevodsky.
\newblock On the zero slice of the sphere spectrum.
\newblock {\em Tr. Mat. Inst. Steklova}, 246(Algebr. Geom. Metody, Svyazi i
  Prilozh.):106--115, 2004.

\bibitem{MR2804268}
V.~Voevodsky.
\newblock Cancellation theorem.
\newblock {\em Doc. Math.}, (Extra volume: Andrei A. Suslin sixtieth
  birthday):671--685, 2010.

\bibitem{MR3076731}
C.~A. Weibel.
\newblock {\em The {$K$}-book}, volume 145 of {\em Graduate Studies in
  Mathematics}.
\newblock American Mathematical Society, Providence, RI, 2013.
\newblock An introduction to algebraic $K$-theory.

\end{thebibliography}
\bibliographystyle{abbrv}

\end{document}